\documentclass[12pt, a4paper, parskip=half, abstracton]{scrartcl}
\usepackage{graphicx} 
\title{A Comparison of Takai and Treumann Dualities}
\author{Vikram Nadig\thanks{Universit\"at Bielefeld\\vikram.nadig@math.uni-bielefeld.de}}
\date{\today}
\usepackage{amssymb}
\usepackage{tikz-cd}
\usepackage{quiver}
\usepackage[hidelinks]{hyperref}
\usepackage[
backend=biber,
style=alphabetic,
]{biblatex}
\usepackage{amsthm}
\usepackage{relsize}
\usepackage{amsfonts}
\usepackage{amsmath}
\usepackage{mathtools}
\usepackage{yhmath}
\numberwithin{equation}{section}
\newtheorem{theorem}{Theorem}[section] 
\newtheorem{prop}[theorem]{Proposition}
\newtheorem{lem}[theorem]{Lemma}

\newtheorem{ddd}[theorem]{Definition}
\newtheorem{kor}[theorem]{Corollary}

\theoremstyle{remark}
\theoremstyle{definition}

\newtheorem{rem}[theorem]{Remark}
\begin{document}

\maketitle

\begin{abstract}
We prove a comparison result between two duality statements - Takai duality, which is implemented by the crossed product functor $- \rtimes G\colon KK^{G} \to KK^{\hat G}$ on equivariant Kasparov categories; and Treumann duality, which asserts the existence of an exotic equivalence of stable $\infty$-categories $\text{Fun}(BG,\text{Perf}(KU_p))\simeq \text{Fun}(B\hat G,\text{Perf}(KU_p))$ given by tensoring with a particular $(G,\hat G)$-bimodule $\textbf{M}$.
\end{abstract}
\setcounter{tocdepth}{1}
\tableofcontents
\section{Introduction}

Let $G$ be a finite abelian $p$-group and $\hat G$ be its Pontryagin dual. In this paper, we study and compare two different duality results on algebraic structures with actions of $G$ and $\hat G$.

The first of these is in the realm of $C^{*}$-algebras, and holds in the greater generality of locally compact abelian groups. Given a $G$-$C^{*}$-algebra (a $C^*$-algebra together with a continuous $G$-action) $A$, we can form the crossed product $A \rtimes G$, which is a $C^{*}$-algebra that acquires a natural action of $\hat G$. It is natural to ask, via the Pontryagin duality isomorphism, what the $G$-$C^{*}$-algebra $(A \rtimes G) \rtimes \hat G$ is. This question was answered by Takai in the 1970s:
\begin{theorem}[\textbf{Takai Duality}]\label{primo}
    There is a $G$-equivariant isomorphism of $C^{*}$-algebras  $$(A \rtimes G) \rtimes \hat G \cong A \otimes K(L^2(G))$$
\end{theorem} Here, $K(L^2(G))$ is given the $G$-action of conjugation with the left-regular representation of $G$ on $L^2(G)$.

The other duality result we take up is in the modern realm of higher algebra. Let $KU_p$ denote the $p$-completed complex $K$-theory spectrum, and let $KU_p[G]$ denote the commutative ring spectrum $KU_p \wedge \Sigma_{+}^{\infty} G$. One can form the stable $\infty$-category of modules $\text{Mod}(KU_p[G]) \simeq \text{Fun}(BG, \text{Mod}(KU_p))$ over $KU_p[G]$, and consider the full subcategory $\text{Fun}(BG, \text{Perf}(KU_p))$ of functors valued in perfect modules. In \cite{treumann2015representations}, a module $\textbf{M}$ over $KU_p[G \times \hat{G}]$ is constructed, whose underlying $KU_p$-module is $KU_p$ itself. Tensoring by $\textbf{M}$ over $KU_p[G]$ and $KU_p[\hat G]$, respectively, induces functors $\text{Fun}(BG,\text{Perf}(KU_p)) \to \text{Fun}(B\hat G,\text{Perf}(KU_p))$ and $\text{Fun}(B\hat G,\text{Perf}(KU_p)) \to \text{Fun}(BG,\text{Perf}(KU_p))$. Treumann proved the following duality result\footnote{We specialize Treumann's result to $KU_p$.} (Theorem 1.5.1 of \cite{treumann2015representations}):
\begin{theorem}[\textbf{Treumann Duality}]\label{secondo}
The functor
\[\begin{tikzcd}
	{\textup{Fun}(BG,\textup{Perf}(KU_p))} && {\textup{Fun}(B\hat G,\textup{Perf}(KU_p))}
	\arrow["({-\otimes_{KU_p[G]}\textup{\textbf{M}}})_p", from=1-1, to=1-3]
\end{tikzcd}\] is an equivalence with inverse $(- \otimes_{KU_p[\hat G]} \textup{\textbf{M}})_p$.
\end{theorem}
This theorem produces an equivalence of stable $\infty$-categories $\text{Fun}(BG,\text{Perf}(KU_p)) \simeq \text{Fun}(B\hat G,\text{Perf}(KU_p))$ that is not induced by any choice of isomorphism $G \cong \hat{G}$ (for example, the $KU_p[G]$-module $KU_p$ with trivial $G$-action is mapped to $KU_p[\hat G]$ under this equivalence).

The article \cite{bunke2023stable} bridges these two different worlds. A stable $\infty$-category $KK^G$ is constructed, which represents Kasparov's $G$-equivariant $KK$-theory for $C^{*}$-algebras, in the following sense. Denote by $KK^{G}_{*}(-,-)$ Kasparov's original equivariant $KK$-groups. There is a functor $kk^{G}\colon GC^{*}\text{Alg}^{nu} \to KK^G$ (where $GC^{*}\text{Alg}^{nu}$ is the category of $G$-$C^{*}$-algebras) such that $$\pi_*KK^{G}(kk^{G}(A),kk^{G}(B)) \cong KK^{G}_*(A,B)$$ for all separable $G$-$C^{*}$-algebras $A$ and $B$. That is, for separable algebras,  Kasparov's classical $G$-equivariant $KK$-groups can be read off from the stable homotopy groups of the mapping spectra in $KK^{G}$. In Section \ref{twoone}, we recall the explicit construction of $KK^{G}$ and $kk^G$ by a series of Dwyer-Kan localizations and state their universal properties.

In the language of equivariant $KK$-theory, Theorems \ref{primo} and \ref{secondo} are similar statements. To motivate why, we consider a simple example. Consider the $C^{*}$-algebra $\mathbb C$ with trivial $G$-action. Its $p$-completed $K$-theory spectrum is nothing but $KU_p$ itself, and it inherits the trivial $G$-action. On the other hand, the crossed product $\mathbb{C} \rtimes G$ can be identified, under the Fourier transform, with $C_0(\hat G)$. This can be identified with a $\hat G$-indexed direct sum $\bigoplus_{\hat G} \mathbb{C}$, and $\hat G$ acts on it by permuting the summands. Consequently, the $p$-completed $K$-theory spectrum of $\mathbb{C} \rtimes G$ is $KU_p[\hat G]$ with the usual $\hat{G}$-action, and so one recovers the association $KU_p \mapsto KU_p[\hat G]$ of Theorem \ref{secondo} by taking the crossed product as illustrated.

In Section \ref{twotwo}, we observe that the crossed product gives a functor $ -\rtimes G\colon GC^{*}\text{Alg}^{nu} \to \hat GC^{*}\text{Alg}^{nu}$. We subsequently show that it descends essentially uniquely to a functor $- \rtimes G\colon KK^{G} \to KK^{\hat G}$. Here, ``descends" means that we have a commutative square of functors 
\[\begin{tikzcd}
	{GC^{*}\text{Alg}^{nu}} && {\hat GC^{*}\text{Alg}^{nu}} \\
	\\
	{KK^{G}} && {KK^{\hat G}}
	\arrow["{\rtimes G}", from=1-1, to=1-3]
	\arrow["{kk^G}"', from=1-1, to=3-1]
	\arrow["{\rtimes G}"', from=3-1, to=3-3]
	\arrow["{kk^{\hat G}}", from=1-3, to=3-3]
\end{tikzcd}\]

If we aspire to formulate a reasonable comparison statement between the two duality results stated above, we must begin by stating Theorem \ref{primo} in categorical language, which we do in Section \ref{three}. From this and the $K_G$-stability of $KK$-theory, we conclude that $- \rtimes G\colon KK^{G} \to KK^{\hat G}$ is an equivalence (Corollary \ref{eq}). Expressed in this way, Takai duality takes on a form similar to Treumann duality.

Taking $K$-theory and $p$-completing yields a natural functor $$r_G\colon KK^{G} \to \text{Mod}(KU_p[G])$$
We define $KK^{G,ft}$ to be the full subcategory of $KK^G$ that is mapped by $r_G$ to $\text{Fun}(BG,\text{Perf}(KU_p))$. We show that $KK^{G,ft}$ is, in fact, a thick subcategory of $KK^{G}$. Further, we show that the crossed product functor $- \rtimes G$ restricts to an equivalence $KK^{G,ft} \simeq KK^{\hat G,ft}$ (Corollary \ref{hari}).

We conclude the paper with our comparison result (Theorem \ref{4.9}):
\begin{theorem}\label{trid}
    There is a natural equivalence filling the square of functors 
\[\begin{tikzcd}
	{KK^{G,ft}} &&& {KK^{\hat G,ft}} \\
	\\
	\\
	{\textup{Fun}(BG,\textup{Perf}(KU_p))} &&& {\textup{Fun}(B\hat G,\textup{Perf}(KU_p))}
	\arrow["{-\rtimes G}" ,"\sim"', from=1-1, to=1-4]
	\arrow["{r_G}"', from=1-1, to=4-1]
	\arrow["{r_{\hat G}}", from=1-4, to=4-4]
	\arrow["\sim","{(-\otimes_{KU_p[G]}\textup{\textbf{M}})_p}"', from=4-1, to=4-4]
\end{tikzcd}\]
\end{theorem}
Both the horizontal arrows are equivalences - the upper one is that of Takai duality, and the lower one is that of Treumann duality.

\textit{Acknowledgements: The author is very grateful to Ulrich Bunke for proposing this problem, his extensive help with ideas and suggestions, and reviewing preliminary drafts in detail. The author also thanks Benjamin D\"unzinger and Fabian Hebestreit for helpful discussions, Christoph Winges for explaining a proof of Proposition \ref{long}., and Anupam Datta and Paul Arne \O stv\ae r for useful comments. The author is also very thankful to the referee for suggesting numerous improvements and explaining a proof of Proposition \ref{4.1}. The author also thanks the editors for pointing out the relevant reference \cite{classically}.}

\textit{The author was supported by the German Research Foundation (DFG) through the collaborative research centre “Integral structures in Geometry and
Representation Theory” (grant no. TRR 358–491392403) at the University of Bielefeld. The author was previously supported by scholarships from ERASMUS+ and the ALGANT consortium.}
\section{The Crossed Product Functor}
Throughout this section, we denote by $G$ a locally compact, second countable, abelian group with identity element $e$ and its Pontryagin dual by $\hat{G}$. The dual group $\hat G$ is endowed with the compact-open topology. We fix Haar measures $\mu$ and $\hat \mu$ on $G$ and $\hat G$, respectively. These are unique up to multiplication by a positive scalar.
\subsection{A stable $\infty$-category for equivariant  $KK$-theory}\label{twoone}
In this subsection, we will introduce the definition and state the universal property of the stable $\infty$-category $KK^G$ referenced in the introduction. In the case that $G$ is discrete, a reference is \cite{bunke2023stable}. In the generality of locally compact groups, however, we follow \cite[Sec. 2.1]{bunke2023ktheory} which asserts that things work similarly, and mentions that a reference is in preparation.

We will denote by $C^{*}\text{Alg}$ the category with objects the unital $C^{*}$-algebras and morphisms the unital $*$-algebra homomorphisms, whereas $C^*\text{Alg}^{nu}$ will refer to the category with objects the $C^{*}$-algebras (not necessarily unital) and morphisms the $*$-algebra homomorphisms\footnote{The notation $C^{*}\text{Alg}^{nu}$ is not to be confused with that of nuclear $C^{*}$-algebras. We are following the convention in \cite{bunke2023stable}.}. We also recall the notion of a pre-$C^{*}$-algebra - it is a $*$-algebra over $\mathbb C$ each element of which has a finite maximal norm\footnote{Let $A$ be a $*$-algebra over $\mathbb C$. For an element $a$ of $A$, we define its maximal norm to be $\text{sup}||\rho(a)||$, where the supremum is taken over all representations $\rho\colon A \to B$, where $B$ is a $C^{*}$-algebra.}. One has a category $C^{*}_{\text{pre}}\text{Alg}^{nu}$ of pre-$C^{*}$-algebras and $*$-algebra homomorphisms. The fully faithful inclusion $C^{*}\text{Alg}^{nu} \hookrightarrow C^{*}_{\text{pre}}\text{Alg}^{nu}$ admits a left adjoint $\hypertarget{shreehari}{\text{compl}}\colon C^{*}_{\text{pre}}\text{Alg}^{nu} \to C^{*}\text{Alg}^{nu}$, which acts on objects by completing with respect to the maximal norm (see Section 2 of \cite{bunke2024kk}).

We denote by $GC^*\text{Alg}^{nu}$ the following category. The objects are pairs of the form $(A,\alpha)$, where $A$ is an object of $C^{*}\text{Alg}^{nu}$ and $\alpha\colon G \to \text{Aut}_{C^{*}\text{Alg}^{nu}}(A,A)$ is a group homomorphism, such that for each element $a$ of $A$, the function $G \to A$ given by $g \mapsto \alpha_g(a)$ is continuous. At times we will be sloppy and simply write $A$ instead of $(A,\alpha)$. The morphisms of $GC^*\text{Alg}^{nu}$ are the $G$-equivariant $C^{*}$-algebra homomorphisms. When $G$ is discrete, $GC^{*}\text{Alg}^{nu}$ can be identified with the functor category $\text{Fun}(BG,C^{*}\text{Alg}^{nu})$. The category $GC^{*}\text{Alg}^{nu}$ has two different symmetric monoidal structures given by the minimal and maximal tensor products. However, in this paper, for simplicity, we will work exclusively with the maximal tensor product, which we denote simply by $-\otimes-$.  

Let $GC^*\text{Alg}^{nu}_{sep}$ denote the full subcategory of $GC^*\text{Alg}^{nu}$ consisting of the separable algebras.
\begin{ddd}\label{fund} Let $D$ be an $\infty$-category, and $F\colon GC^*\textup{Alg}^{nu}_{sep} \to D$ a functor.
\begin{enumerate}
\item The functor $F$ is said to be homotopy invariant if it sends homotopy equivalences to equivalences. The notion of homotopy is defined with respect to the point-norm topology on the mapping sets in $GC^*\textup{Alg}^{nu}_{sep}$.
\item Let $K_G\coloneq K(L^2(G) \otimes l^2) \in GC^*\textup{Alg}^{nu}$ be the algebra of compact operators on the Hilbert space $L^2(G) \otimes l^2$, with group action induced by conjugation of the natural action on $L^2(G)$. A morphism $f\colon A \to B$ in $GC^*\textup{Alg}^{nu}_{sep}$ is a $K_G$-equivalence if $f \otimes K_G$ is a homotopy equivalence.
\item The functor $F$ is said to be $G$-stable if for every equivariant isometric embedding of separable unitary\footnote{Here, unitary means that $G$ acts by unitaries on $V'$ and $V$.} $G$-Hilbert spaces $V' \hookrightarrow V$, with $V' \neq 0$, and for all objects $A$ of $GC^{*}\text{Alg}^{nu}$, the induced map $F(A \otimes K(V')) \to F(A \otimes K(V))$ is an equivalence.
\end{enumerate}
\end{ddd}
\begin{prop}\label{eqstab}
    A homotopy invariant functor $F$ inverts $K_G$-equivalences if and only if it is $G$-stable.
\end{prop}
\begin{proof}
    First, suppose that $F$ is $G$-stable. Let $f\colon A \to B$ be a $K_{G}$-equivalence. Write $\hat K_G$ for $K(\mathbb C \oplus (L^2(G) \otimes l^2))$ with the obvious $G$-action, so that one has equivariant isometric inclusions $\mathbb C \hookrightarrow \hat K_G$ and $K_G \hookrightarrow \hat K_G$. This gives a commutative diagram 
\[\begin{tikzcd}
	A && {A \otimes \hat K_G} && {A \otimes K_G} \\
	B && {B \otimes \hat K_G} && {B \otimes K_G}
	\arrow[from=1-1, to=1-3]
	\arrow[from=1-5, to=1-3]
	\arrow["f", from=1-1, to=2-1]
	\arrow["{f \otimes \hat K_G}", from=1-3, to=2-3]
	\arrow["{f \otimes K_G}", from=1-5, to=2-5]
	\arrow[from=2-1, to=2-3]
	\arrow[from=2-5, to=2-3]
\end{tikzcd}\]
By assumption, $F$ inverts all the horizontal arrows and the right vertical arrow (the latter because $f \otimes K_G$ is a homotopy equivalence and $F$ is homotopy invariant). Therefore, it inverts $f$.

Conversely, suppose that $F$ inverts $K_{G}$-equivalences. It suffices to show that $K(V') \to K(V)$ is a $K_{G}$-equivalence, where $V'$ and $V$ are as in Definition \ref{fund}.3.

Let $K$ denote $K(l^2)$. Suppose the $G$-actions on $V$ and $V'$ are trivial. Then $K(V') \otimes K \to K(V) \otimes K$ is isomorphic to a left upper corner inclusion, and therefore becomes a homotopy equivalence on tensoring again by $K$. As $K_G \cong K_G \otimes K \otimes K$, we conclude the result in this case.

In general, given a Hilbert space $H$ with $G$-action $\rho$, we can identify $L^2(G) \otimes H$ with $L^2(G,H)$ (the $G$-action given on the latter by $g(f(h))\coloneq \rho_{g}f(g^{-1}h)$), and in particular, $L^2(G) \otimes \text{Res}^{G}(H)$ with $L^2(G,\text{Res}^{G}(H))$, where $\text{Res}^{G}(H)$ is the Hilbert space $H$ with trivial $G$-action. There is a natural equivariant isomorphism $\phi_H\colon L^2(G) \otimes H \cong L^2(G) \otimes \text{Res}^{G}(H)$ given by the formula $f(h) \mapsto \rho(h)^{-1}f(h)$, which in turn induces an isomorphism $K(H) \otimes K_G \cong \text{Res}^{G}(K(H)) \otimes K_{G}$, and so we reduce to the previous case.
\end{proof}
We write $\hypertarget{hom}{GC^{*}\text{Alg}^{nu}_{h,sep}}$ for the localization\footnote{We always refer to the Dwyer-Kan localization when we write ``localization". See \cite[\href{https://kerodon.net/tag/01N0}{Tag 01N0}]{kerodon} and \cite[\href{https://kerodon.net/tag/01N1}{Tag 01N1}]{kerodon} for an existence and uniqueness statement.} of $GC^{*}\text{Alg}^{nu}_{sep}$ by the homotopy equivalences, and $\hypertarget{stable}{L_{K_G}GC^*\text{Alg}^{nu}_{h,sep}}$ for the localization of $GC^{*}\text{Alg}^{nu}_{sep}$ by the homotopy equivalences and the $K_G$-equivalences. We have a sequence of localizations 
\[\begin{tikzcd}
	{L_{h,K_G}\colon GC^*\text{Alg}^{nu}_{sep}} && {GC^*\text{Alg}^{nu}_{h,sep}} && {L_{K_G}GC^*\text{Alg}^{nu}_{h,sep}}
	\arrow["{L_h}", from=1-1, to=1-3]
	\arrow["{L_{K_G}}", from=1-3, to=1-5]
\end{tikzcd}\]
\begin{ddd} Let $D$ be a pointed $\infty$-category, and $F\colon GC^*\textup{Alg}^{nu}_{sep} \to D$ a reduced functor (that is, $F(0)$ is a zero object).
 An exact sequence $0 \to A \to B \to C \to 0$ in $GC^*\textup{Alg}^{nu}_{sep}$ is semi-split if there exists a completely positive and contractive $G$-equivariant section $C \to B$. The functor $F$ is said to be semi-exact if it takes semi-split exact sequences to fibre sequences. 
\end{ddd}
We remark that completely analogous definitions can be made for $GC^{*}\text{Alg}^{nu}$ in place of $GC^*\text{Alg}^{nu}_{sep}$.

If $$\mathcal{E} \colon 0 \to A \to B \xrightarrow{f} C \to 0$$ is an exact sequence in $GC^{*}\text{Alg}^{nu}$, there is a natural induced map $i_f\colon A \to C(f)$, where $C(f)$ is the cone on $f$. Let $W'$ denote the class of morphisms of the form $L_{h,K_G}(i_f)$ taken over all semi-split exact sequences $\mathcal{E}$. Write $W$ for the closure of $W'$ by the 2-out-of-3 property and pullbacks.

Let $$L_W\colon L_{K_G}GC^*\text{Alg}^{nu}_{h,sep} \to L_{K_G}GC^*\text{Alg}^{nu}_{h,sep}[W^{-1}]$$ be the localization by the class $W$. The $\infty$-category $L_{K_G}GC^*\text{Alg}^{nu}_{h,sep}[W^{-1}]$ is left-exact and semi-additive (see \cite[Cor. 4.3.2]{bunke2024kk} and \cite[Prop. 4.4]{bunke2024kk} in the case that $G$ is trivial and item 2 at the beginning of \cite[Sec. 3.4]{bunke2024etheory} for discrete $G$). There is a Bousfield localization $$\text{incl}\colon(L_{K_G}GC^*\text{Alg}^{nu}_{h,sep}[W^{-1}])^{\text{grp}} \leftrightarrows L_{K_G}GC^*\text{Alg}^{nu}_{h,sep}[W^{-1}]\colon \Omega^2$$ where the superscript grp indicates the full subcategory of group objects (see \cite[Prop. 7.2]{bunke2024kk} in the case that $G$ is trivial, and item 4 at the beginning of \cite[Sec 3.4]{bunke2024etheory} for an analogous statement for $E$-theory (for discrete $G$)).

We define the $\infty$-category $KK^{G}_{sep} \coloneq (L_{K_G}GC^*\text{Alg}^{nu}_{h,sep}[W^{-1}])^{\text{grp}}$, and the functor $kk^G_{sep}$ as the composite 
\[\begin{tikzcd}
	{GC^*\text{Alg}^{nu}_{sep}} && {L_{K_G}GC^*\text{Alg}^{nu}_{h,sep}} && {L_{K_G}GC^*\text{Alg}^{nu}_{h,sep}[W^{-1}]} && {KK^{G}_{sep}}
	\arrow["{L_{h,K_G}}", from=1-1, to=1-3]
	\arrow["{L_{W}}", from=1-3, to=1-5]
	\arrow["{\Omega^2}", from=1-5, to=1-7]
\end{tikzcd}\]

The $\infty$-category $KK^{G}_{sep}$ is characterized essentially uniquely by the following universal property. 
\begin{theorem}
$KK^G_{sep}$ is a stable $\infty$-category, the functor $kk^G_{sep}$ is homotopy invariant, $G$-stable and semi-exact; and for any stable $\infty$-category $D$, pre-composition by $kk^G_{sep}$ induces an equivalence of functor categories$$ \textup{Fun}^{ex}(KK^G_{sep},D) \simeq \textup{Fun}^{h,Gs,se}(GC^*\text{Alg}^{nu}_{sep}, D)$$The superscripts $\{h,Gs,se\}$ refer to the full subcategory of homotopy invariant, $G$-stable, semi-exact functors; the superscript $ex$ denotes the full subcategory of exact funtors.
\end{theorem}

We define $KK^G \coloneq \text{Ind}(KK^G_{sep})$ as the $\text{Ind}$-completion, with $y\colon KK^G_{sep} \to KK^G$ the Yoneda embedding. We define the functor $kk^G\colon GC^*\text{Alg}^{nu} \to KK^G$ by left-Kan extension as indicated below:
\[\begin{tikzcd}
	{GC^*\text{Alg}^{nu}_{sep}} && {GC^*\text{Alg}^{nu}} && {KK^G} \\
	\\
	&& {KK^G_{sep}}
	\arrow[hook, from=1-1, to=1-3]
	\arrow["{kk^G}", dashed, from=1-3, to=1-5]
	\arrow["{kk^G_{sep}}"', from=1-1, to=3-3]
	\arrow["y"', hook, from=3-3, to=1-5]
\end{tikzcd}\]
We need another preliminary notion to characterize $KK^G$.
\begin{ddd} Let $D$ be a cocomplete $\infty$-category, and $F\colon GC^*\textup{Alg}^{nu} \to D$ a functor.
         $F$ is said to be $s$-finitary if the natural map $$\textup{colim}_{A' \subset_{sep} A}F(A') \to F(A)$$ is an equivalence, for all objects $A$ of $GC^*\textup{Alg}^{nu}$. The colimit here is taken over all separable $G$-$C^{*}$-subalgebras of $A$. 
\end{ddd} The $\infty$-category $KK^{G}$ is determined essentially uniquely by the following universal property (see \cite[Eq. 2.2]{bunke2023ktheory}, and \cite[Thm. 3.3]{bunke2023stable} for a proof in the discrete case\footnote{One uses the fact that homotopy invariance is automatic from the other conditions when $G$ is discrete, see \cite[Rem. 2.3]{bunke2023stable}}).

\begin{theorem}{$KK^G$ is a stable and cocomplete $\infty$-category, the functor $kk^G$ is homotopy invariant, $G$-stable, semi-exact and $s$-finitary; and for any stable, cocomplete $\infty$-category $D$, pre-composition by $kk^G$ induces an equivalence $$ \textup{Fun}^{colim}(KK^G,D) \simeq \textup{Fun}^{h,Gs,se,sfin}(GC^*\textup{Alg}^{nu}, D)$$}The superscripts $\{h,Gs,se, sfin\}$ refer to the full subcategory of homotopy invariant, $G$-stable, semi-exact and $s$-finitary functors; the superscript $colim$ denotes the full subcategory of colimit-preserving funtors.
\end{theorem}

We remark that $KK^G$ also admits a symmetric monoidal structure and that $kk^G$ can be refined to a symmetric monoidal functor, and an analogous universal property can be stated in this situation. See \cite[Prop. 3.8]{bunke2023stable} for a proof in the discrete case, \cite[Eq. 2.3]{bunke2023ktheory} for a formulation of the universal property and \cite[Thm. 8.5]{bunke2024kk} for a proof thereof in the non-equivariant setting. 

We denote $KK^{\{e\}}$ simply by $KK$, where $\{e\}$ is the trivial group. For an object $A$ of $KK$, we write $K(A)$ ``the $K$-theory spectrum of $A$" for the mapping spectrum $KK(\mathbb C,A)$. We denote the spectrum $K(\mathbb C)$ by $KU$; this coincides with the $2$-periodic topological complex $K$-theory spectrum (see \cite[Rem. 9.19]{bunke2024kk}). 
The spectrum-valued functor $K(-)$ naturally refines to a functor valued in the $\infty$-category $\text{Mod}(KU)$ of modules over the commutative ring spectrum $KU$.

A detailed account of the properties of $K(-)$ is presented in \cite{bunke2023survey}.
\subsection{The Crossed Product Functor}\label{twotwo}
For a $G$-$C^{*}$-algebra $A$, one can consider the set $C_c(G,A)$ of compactly supported continuous $A$-valued functions on $G$. This set acquires the structure of a ${*}$-algebra over $\mathbb C$ with the following definitions of the algebraic operations. For $f,f' \in C_c(G,A)$ and $g \in G$, we set:
\begin{itemize}
    \item $(f+f')(g)\coloneq f(g) + f'(g)$
    \item $(f*f')(g)\coloneq \int_{G}f(h)\alpha_h(f'(h^{-1}g))\mu(h)$
    \item $f^{*}(g)\coloneq \alpha_g(f(g^{-1})^{*})$
\end{itemize}
With this algebraic structure, one can check that $C_c(G,A)$ is a pre-$C^{*}$-algebra - in fact, for $f \in C_c(G,A)$, the maximal norm of $f$ is bounded by its $L_1$-norm. We define the (maximal)\footnote{There is also a variant called the reduced crossed product $A \rtimes_{r} G$ obtained by completing $C_c(G,A)$ with respect to a particular representation. However, we shall work exclusively with the maximal crossed product for simplicity, and because of its compatibility with the maximal tensor product (Lemma \ref{2.5}).} crossed product $A \rtimes G$ to be the completion $\hyperlink{shreehari}{\text{compl}}(C_c(G,A))$\footnote{One can more generally define $A \rtimes G$ even for non-abelian $G$.}. This yields a functor $-\rtimes G\colon GC^*\text{Alg}^{nu} \to C^*\text{Alg}^{nu}$.

We will refine the functor $-\rtimes G$ to take values in $\hat{G}C^*\text{Alg}^{nu}$, and show that it descends essentially uniquely to a functor $KK^G \to KK^{\hat G}$. 
\begin{rem}\label{baajskandalis}
    At the level of $KK$-theory groups, this is classical work by Baaj and Skandalis. They construct (cf. \cite[Thm. 6.19]{classically} and note that $\mathbb C \rtimes G \cong C_0(\hat G)$ as $G$ is assumed to be abelian), for separable algebras $A$ and $B$, natural homomorphisms $J_G\colon KK^{G}_{*}(A,B) \to KK^{\hat G}_{*}(A \rtimes G, B \rtimes G)$ and $J_{\hat G} \colon KK^{\hat G}_{*}(A \rtimes G, B \rtimes  G) \to KK^{G}_{*}(A,B)$. The functors $- \rtimes G \colon KK^{G} \to KK^{\hat G}$ and $- \rtimes \hat G \colon KK^{\hat G} \to KK^{G}$ that we will construct in the remainder of this section serve as (higher) categorical refinements to the maps $J_{G}$ and $J_{\hat G}$, respectively.
\end{rem}

To begin with, consider a character $\chi$ in $\hat{G}$, and an element $f$ of $C_c(G,A)$, where $A$ is a $G$-$C^{*}$-algebra. We define $\chi f \in C_c(G,A)$ by pointwise multiplication with the character value: $$(\chi f)(g)\coloneq \chi(g)f(g)$$
\begin{prop}\label{2.1}
    The rule defined above gives a strongly continuous action of $\hat{G}$ on $A \rtimes G$.
\end{prop}
\begin{proof}
    By substituting into the definitions, one easily checks that we have an action of $\hat G$ on the pre-$C^{*}$-algebra $C_c(G,A)$; applying $\text{compl}$, we get an action of $\hat G$ on $A \rtimes G$. 
    
     We check that the action is continuous. We must show that for each element $f$ of $A \rtimes G$, the map $\hat{G} \to A\rtimes G$ given by $\chi \to \chi f$ is continuous. To this end, we may assume that $f \in C_c(G,A)$, as $C_c(G,A)$ is dense in $A \rtimes G$. Let $S$ be the support of $f$. By assumption, $S$ is compact, so $\mu(S)$ is finite. Suppose $\{\chi_i\}$ is a net converging to the trivial character $1$. As the maximal norm is bounded by the $L_1$-norm, it suffices to show that $||\chi_i f - f||_{L_1}$ goes to $0$. We have $$||\chi_i f - f||_{L_1} = \int_{G}||\chi_if - f||\mu \leq \int_{G}||\chi_{i}-1||||f||\mu$$ 
    To conclude, note that $||f||$ is bounded, and that by definition of the compact-open topology, we can make the difference $||\chi_i - 1||$ arbitrarily small on $S$.
    \end{proof}
    It now follows that:
\begin{prop}\label{cros}
    $-\rtimes G$ determines a functor\footnote{Throughout this paper, we will abuse notation and use $- \rtimes G$ for different functors. By default, we will mean the functor $GC^{*}\text{Alg}^{nu} \to \hat{G}C^{*}\text{Alg}^{nu}$ of Proposition \ref{cros} or the functor $KK^{G} \to KK^{\hat G}$ of Theorem \ref{twofin}, as should be clear from the context. In other cases, we will explicitly indicate the source and target.} $GC^*\textup{Alg}^{nu} \to \hat{G}C^*\textup{Alg}^{nu}$.
\end{prop}

We record the following important categorical fact.

\begin{prop}\label{colim}
    The functor $-\rtimes G\colon GC^*\textup{Alg}^{nu} \to \hat{G}C^*\textup{Alg}^{nu}$ preserves filtered colimits.
\end{prop}
\begin{proof}
As this result appears to be folklore, we only sketch the idea.

Let $I$ be a small filtered category, and $A\colon I \to GC^{*}\text{Alg}^{nu}$ a functor. We first assume that $A$ is unital - i.e., it factors through the inclusion $GC^*\textup{Alg} \hookrightarrow GC^*\textup{Alg}^{nu}$. One can show by hand that the canonical map $\text{colim}_I(A \rtimes G) \to (\text{colim}_IA) \rtimes G$ is surjective. To produce an inverse, we use the observation that if $A_i \to A_j$ is a unital homomorphism, then $A_i \rtimes G \to A_j \rtimes G$ is an essential homomorphism and thus functorially induces a map on multiplier algebras $M(A_i \rtimes G) \to M(A_j \rtimes G)$ together with \cite[Prop. 2.34]{algwill}. The extension from the unital to the non-unital case is carried out exactly as in the proof of \cite[Thm. 7.13.1]{bunke2023stable}.
\end{proof}
We now proceed to show that $- \rtimes G\colon GC^{*}\text{Alg}^{nu} \to \hat{G}C^{*}\text{Alg}^{nu}$ descends to a functor $- \rtimes G\colon KK^G \to KK^{\hat G}$, by showing that it descends through the intermediate localizations introduced in Section \ref{twoone}. Recall the $\infty$-categories $\hyperlink{hom}{GC^{*}\text{Alg}^{nu}_{h,sep}}$ and $\hyperlink{stable}{L_{K_G}GC^{*}\text{Alg}^{nu}_{h,sep}}$ introduced there.

Observe that the functor $-\rtimes G\colon GC^*\text{Alg}^{nu} \to \hat{G}C^*\text{Alg}^{nu}$ preserves separable algebras. We recall the following criterion to check for topological enrichment.
\begin{lem}\label{2.3}
    Let $F\colon GC^{*}\textup{Alg}^{nu}_{sep} \to G'C^{*}\textup{Alg}^{nu}_{sep}$ be a functor. Suppose that for all commutative algebras $B$ and for all objects $A$ of $GC^{*}\textup{Alg}^{nu}_{sep}$, there exists a natural isomorphism  $$F(A \otimes B) \cong F(A) \otimes B $$ (where $B$ is given trivial group actions). Further, assume that the composite isomorphism $$ F(A) \cong F(A \otimes \mathbb C) \cong F(A) \otimes \mathbb C \cong F(A)$$ is the identity for all objects $A$ of  $GC^{*}\textup{Alg}^{nu}_{sep}$. Then $F$ descends essentially uniquely to a functor $F\colon GC^*\textup{Alg}^{nu}_{h,sep} \to G'C^*\textup{Alg}^{nu}_{h,sep}$.
\end{lem}
\begin{proof}
    Write $\text{Hom}_{G}(-,-)$ and $\text{Hom}_{G'}(-,-)$ for the mapping spaces in $GC^{*}\textup{Alg}^{nu}_{sep}$ and $G'C^{*}\textup{Alg}^{nu}_{sep}$, respectively.

    For objects $A, B$ of $GC^{*}\textup{Alg}^{nu}_{sep}$, and a compact topological space $X$, consider the composite map \begin{align*}
    \text{Hom}_{\textbf{Top}}(X,\text{Hom}_G(A,B)) &\cong \text{Hom}_{G}(A, B \otimes C(X))\\
     & \to \text{Hom}_{G}(F(A), F(B \otimes C(X)))\\
     & \cong \text{Hom}_{G'}(F(A), F(B) \otimes C(X))\\
     & \cong \text{Hom}_{\textbf{Top}}(X,\text{Hom}_{G'}(F(A),F(B)))
     \end{align*}
     The condition on the isomorphism in the statement of the lemma implies that the underlying map of sets $\text{Hom}_{\textbf{Top}}(X,\text{Hom}_G(A,B)) \to \text{Hom}_{\textbf{Top}}(X,\text{Hom}_{G'}(F(A),F(B)))$ constructed above is induced by $F$. This shows continuity with respect to the compactly generated topology, which suffices to conclude the desired statement.
\end{proof}
\begin{lem}\label{2.5}
Let $A$ be an object of $GC^{*}\textup{Alg}^{nu}$ and $B$ an object of $C^{*}\textup{Alg}^{nu}$. There is a natural isomorphism of algebras $$\phi_{A,B}\colon(A \rtimes G) \otimes B \simeq (A \otimes B) \rtimes G$$ given by $f \otimes b \mapsto (g \mapsto f(g) \otimes b)$ on the dense subset $C_c(G,A) \otimes B$. 
\end{lem}
\begin{proof}
    See \cite[Lem. 2.75]{algwill}.
\end{proof}
\begin{prop}\label{omo}
    The functor $-\rtimes G\colon GC^*\textup{Alg}^{nu}_{sep} \to \hat{G}C^*\textup{Alg}^{nu}_{sep}$ descends essentially uniquely to a functor $-\rtimes G\colon GC^*\textup{Alg}^{nu}_{h,sep} \to \hat{G}C^*\textup{Alg}^{nu}_{h,sep}$.
\end{prop}
\begin{proof}
    We check the condition of Lemma \ref{2.3}. Indeed, the isomorphism $\phi_{A,B}$ of Lemma \ref{2.5} gives us the desired isomorphism: it is $\hat G$-equivariant, and the composition $$A \rtimes G \cong (A \rtimes G) \otimes \mathbb C\xrightarrow{\phi_{A,\mathbb C}} (A \otimes \mathbb C) \rtimes G \cong A \rtimes G$$ is the identity.
\end{proof}
\begin{lem}\label {2.6}
    Let $B$ be a $C^*$-algebra, and let $\rho\colon G \to U(M(B))$ be a strongly continuous group homomorphism. Then $B$ becomes a $G$-$C^{*}$-algebra with action $\beta_g(b) \coloneq \rho_g b \rho_{g^{-1}}$. Denote by $\underline B$ the algebra $B$ with trivial $G$-action. For any $G$-$C^{*}$-algebra $(A,\alpha)$, there is a natural (in $A$) algebra isomorphism $$\psi_{A,B}\colon(A \otimes B) \rtimes G \to (A \otimes \underline B) \rtimes G$$ given by $f(g) \mapsto f(g)(\textup{id}\otimes \rho_g)$ on the dense subset $C_c(G,A \otimes B)$. 
\end{lem}
\begin{proof}
    Note that $A \otimes B$ is an ideal in $M(A) \otimes M(B)$, so by the universal property of the multiplier algebra, we obtain a unital map $M(A)\otimes M(B) \to M(A \otimes B)$. We can therefore define $\textup{id} \otimes \rho_g$ as the composite 
\[\begin{tikzcd}
	G && {U(M(A)) \times U(M(B))} && {U(M(A) \otimes M(B))} & {U(M(A \otimes B))}
	\arrow["{(\text{id},\rho_g)}", from=1-1, to=1-3]
	\arrow["{(a,b) \mapsto a \otimes b}", from=1-3, to=1-5]
	\arrow[from=1-5, to=1-6]
\end{tikzcd}\]
    We check that the formula $f(g) \mapsto f(g)(\textup{id}\otimes \rho_g)$ defines a $*$-algebra isomomorphism $$\psi'_{A,B}\colon C_c(G,A \otimes B) \to C_c(G,A \otimes \underline{B})$$
    with inverse $\psi'^{-1}_{A,B}$ given by $f(k) \mapsto f(k)(\text{id} \otimes \rho_{k^{-1}})$.
    
    We now set $\psi_{A,B} \coloneq \text{compl}(\psi'_{A,B})$.
\end{proof}
\begin{prop}\label{stabledesc}
    The functor $-\rtimes G\colon GC^*\textup{Alg}^{nu}_{h,sep} \to \hat{G}C^*\textup{Alg}^{nu}_{h,sep}$ descends essentially uniquely to a functor $-\rtimes G\colon L_{K_G}GC^*\textup{Alg}^{nu}_{h,sep} \to L_{K_{\hat G}}\hat{G}C^*\textup{Alg}^{nu}_{h,sep}$.
\end{prop}
\begin{proof}
We will show that the composition $$F\colon GC^*\text{Alg}^{nu} _{sep}\xrightarrow{- \rtimes G} \hat{G}C^*\text{Alg}^{nu}_{sep} \xrightarrow{L_{K_{\hat G},h}} L_{K_{\hat G}}\hat{G}C^*\text{Alg}^{nu}_{h,sep}$$ is $G$-stable. Since $F$ is homotopy invariant by Lemma \ref{omo}, this would imply the assertion by the universal property of $L_{K_G}GC^*\textup{Alg}^{nu}_{h,sep}$.

    Let $V'$ and $V$ be $G$-Hilbert spaces together with an equivariant unitary embedding $V' \hookrightarrow V$.  We have to show that for any $A \in GC^*\text{Alg}^{nu}_{sep}$, the induced map $A \otimes K(V') \to A \otimes K(V)$ is sent to an equivalence by $F$.
    
    Using Lemma \ref{2.5} and Lemma \ref{2.6}, we have the following commutative diagram of algebras, with the horizontal arrows being natural isomorphisms:
\[\begin{tikzcd}
	{(A \otimes K(V'))\rtimes G} && {(A \otimes \underline{K(V')})\rtimes G} && {(A\rtimes G)\otimes \underline{K(V')}} \\
	\\
	{(A \otimes K(V)) \rtimes G} && {(A \otimes \underline{K(V)})\rtimes G} && {(A\rtimes G)\otimes \underline{K(V)}}
	\arrow[from=1-1, to=3-1]
	\arrow["{\psi_{A,K(V')}}", from=1-1, to=1-3]
	\arrow["{\psi_{A,K(V)}}", from=3-1, to=3-3]
	\arrow[from=1-3, to=3-3]
	\arrow["{\phi_{A,K(V')}}"', from=1-5, to=1-3]
	\arrow["{\phi_{A,K(V)}}"', from=3-5, to=3-3]
	\arrow[from=1-5, to=3-5]
\end{tikzcd}\]
We can upgrade this to a commutative diagram of $\hat G$-algebras by introducing the appropriate $\hat G$ actions on the algebras in the middle and right columns. It is easy to see that this will be the usual crossed product action on the middle column, and $\hat \alpha \otimes \text{id}$ on the right column, where $\hat \alpha$ is the usual action on $A \rtimes G$.

The right vertical map is induced by the identity on $A \rtimes G$ and the embedding $V' \hookrightarrow V$. Therefore, it becomes an equivalence on passage to $L_{K_{\hat G}} \hat{G}C^*\text{Alg}^{nu}_{h,sep}$ (Proposition \ref{eqstab}). Hence, so does the left vertical map, and this is what we wanted to show.
\end{proof}
 
\begin{prop}
    The functor $-\rtimes G\colon  L_{K_G}GC^*\textup{Alg}^{nu}_{sep,h} \to L_{K_{\hat G}}\hat{G}C^*\textup{Alg}^{nu}_{sep,h}$ descends essentially uniquely to an exact functor $-\rtimes G\colon KK^{G}_{sep} \to KK^{\hat G}_{sep}$.
\end{prop}
\begin{proof}
    First, we show that $-\rtimes G\colon GC^*\text{Alg}^{nu}_{sep} \to \hat{G}C^*\text{Alg}^{nu}_{sep}$ preserves semi-exact sequences.
    
    So, suppose that we are given a semi-exact sequence of $G$-$C^{*}$-algebras 
\[\begin{tikzcd}
	0 & I & A & Q & 0
	\arrow[from=1-1, to=1-2]
	\arrow[from=1-2, to=1-3]
	\arrow[from=1-3, to=1-4]
	\arrow[from=1-4, to=1-5]
\end{tikzcd}\] with a $G$-equivariant completely positive and contractive split $r\colon Q \to A$.
It is known (Proposition 9 of \cite{meyerkth}) that this yields a semi-exact sequence of algebras  
\[\begin{tikzcd}
	0 & {I \rtimes G} & {A \rtimes G} & {Q\rtimes G} & 0
	\arrow[from=1-1, to=1-2]
	\arrow[from=1-2, to=1-3]
	\arrow[from=1-3, to=1-4]
	\arrow[from=1-4, to=1-5]
\end{tikzcd}\] with completely positive and contractive split $\hat r\colon Q\rtimes G \to A \rtimes G$ induced by post-composition with $r$ on $C_c(G,Q) \to C_c(G,A)$, easily seen to also be $\hat G$-equivariant.

Thus, the composition $kk^{\hat G} \circ \rtimes G\colon GC^{*}\text{Alg}^{nu} \to KK^{\hat G}_{sep}$ is semi-exact, and we now conclude using the universal property of $KK^G_{sep}$.
\end{proof}
To proceed, we need the notion of an $\text{Ind}$-$s$-finitary functor:
\begin{ddd}\label{2.2694} \textup{(\cite[Def. 4.1]{bunke2023stable})\footnote{This definition and the following two lemmas are stated in \cite{bunke2023stable} for discrete groups. However, the arguments provided there work almost verbatim in our generality.}}\\
A functor $F\colon GC^{*}\textup{Alg}^{nu} \to HC^{*}\textup{Alg}^{nu}$ that preserves separable algebras is said to be $\textup{Ind}$-$s$-finitary if:
        \begin{enumerate}
            \item For all $A \in GC^*\textup{Alg}^{nu}$, the inductive system $F(A')^{F(A)}_{A' \subset_{sep} A}$ is cofinal in the system of all separable algebras of $F(A)$.
            \item The canonical map $F(A')_{A' \subset_{sep} A} \to (F(A')^{F(A)})_{A' \subset_{sep} A}$ is an isomorphism of inductive systems in $\textup{Ind}(HC^{*}\textup{Alg}^{nu})$.
        \end{enumerate}
        The notation $F(A')^{F(A)}$ denotes the image of $F(A')$ in $F(A)$.
\end{ddd}
\begin{lem}\label{2.9}
    Let $F\colon GC^{*}\textup{Alg}^{nu} \to HC^{*}\textup{Alg}^{nu}$ be a functor that preserves separable algebras and satisfies item 1 of Definition \ref{2.2694}. If $F$ preserves countably filtered colimits, then it is $\textup{Ind}$-$s$-finitary.
\end{lem}
\begin{proof}
    This is \cite[Lem. 4.3.1]{bunke2023stable}.
\end{proof}
The notion of $\text{Ind}$-$s$-finitariness is related to that of $s$-finitariness by the following result:
\begin{lem}\label{2.10}
    Let $F'\colon HC^{*}\textup{Alg}^{nu} \to D$ be an $s$-finitary functor, and $F\colon GC^{*}\textup{Alg}^{nu} \to HC^{*}\textup{Alg}^{nu}$ an $\textup{Ind}$-$s$-finitary functor. Then, the composite $F'\circ F\colon GC^{*}\textup{Alg}^{nu} \to D$ is an $s$-finitary functor.
\end{lem}
\begin{proof}
    This is \cite[Lem. 4.2]{bunke2023stable}.
\end{proof}
We can now conclude this section with its main result:
\begin{theorem}\label{twofin}
    The functor $-\rtimes G\colon KK^{G}_{sep} \to KK^{\hat G}_{sep}$ extends essentially uniquely to a colimit-preserving functor $-\rtimes G\colon KK^{G} \to KK^{\hat G}$.
\end{theorem}
\begin{proof}
    By Lemma \ref{2.10}, all we have to check is that $-\rtimes G\colon GC^*\text{Alg}^{nu} \to \hat{G}C^*\text{Alg}^{nu}$ is $\text{Ind}$-$s$-finitary. 

    We first verify the cofinality condition of Definition \ref{2.2694}. Let $A$ be a $G$-$C^{*}$-algebra, and $B$ a separable $\hat{G}$-subalgebra of $A \rtimes G$. Let $\{f_i\}$ be a countable dense subset of $B$. It is a standard fact (see \cite[Lem. 1.87]{algwill}) that elements of the form $f \otimes a$, with $f \in C_c(G)$ and $a \in A$, generate $A \rtimes G$. Thus, for each $i$, one can find a sequence $\{f_{ik} = \Sigma_{l=1}^{n_{k,i}} g_{l,k,i} \otimes a_{l,k,i}\}$ of linear combinations of elementary tensors that converges to $f_i$. Consider the $G$-$C^{*}$-subalgebra $A'$ of $A$ generated by the collection $\{a_{l,k,i}\}$. Then $A'$ is separable and $B$ is contained in the image of $A' \rtimes G$ in $A \rtimes G$.

     By Proposition \ref{colim}, $-\rtimes G\colon GC^*\text{Alg}^{nu}\to \hat{G}C^{*}\text{Alg}^{nu}$ preserves filtered colimits. 

     The theorem now follows from Lemma \ref{2.9}.
\end{proof}
\section{Takai Duality}\label{three}

We retain the notation and conventions from the previous section - we will denote by $G$ a locally compact, second countable, abelian group with identity element $e$ and its Pontryagin dual by $\hat{G}$. The dual group $\hat G$ is endowed with the compact-open topology. We fix Haar measures $\mu$ and $\hat \mu$ on $G$ and $\hat G$, respectively. These are unique up to multiplication by a positive scalar.

Classically, the Takai duality statement is simply stated as the existence of a certain equivariant isomorphism of algebras (Theorem \ref{primo}); however, this is not sufficient for our purposes. We shall begin by restating this in categorical language.
\begin{theorem}
    Let $H\colon GC^*\textup{Alg}^{nu} \to GC^*\textup{Alg}^{nu}$ be the functor $A \mapsto (A \rtimes G) \rtimes \hat{G}$, with the double-dual action, and let $H'\colon GC^*\textup{Alg}^{nu} \to GC^*\textup{Alg}^{nu}$ be the functor $A \mapsto A \otimes K(L^2(G))$ with the tensor action (here $K(L^2(G))$ is given the usual conjugation action). Then there is an isomorphism of functors $F \simeq F'$.
\end{theorem}
\begin{proof}
    The classical statement is that, for any $G$-$C^{*}$-algebra $A$, there exists an equivariant isomorphism $(A \rtimes G) \rtimes \hat{G} \cong A \otimes K(L^2(G))$; and this isomorphism is constructed as a sequence of isomorphisms (see \cite[Sec. 7.1]{algwill} for the proof). We will upgrade this to a sequence of isomorphisms between functors and focus on the naturality, omitting the well-established details of why each map that appears is an equivariant isomorphism.
    
    Let $A$ be a $G$-$C^{*}$-algebra with action $\alpha$. Classically, the major part of the argument is the construction of an algebra isomorphism $$\phi_1^{A}\colon (A \rtimes G) \rtimes \hat{G} \to ({C_0(G)} \otimes \underline{A}) \rtimes G$$ where $\underline A$ is the algebra $A$ with trivial $G$-action. Let $F$ be an element of $C_c(\hat G \times G, A)$, interpreted as an element of $(A \rtimes G) \rtimes \hat{G}$ under the identification $C_c(\hat G \times G, A) \cong C_c(\hat G, C_c(G,A)) \subset C_c(G,A \rtimes G) \subset (A \rtimes G) \rtimes \hat G$. For such an element, $\phi_1^{A}$ is described by the formula $$\phi_1^{A}(F)(s,r)= \int_{\hat G} \alpha_r^{-1}(F(\gamma,s))\overline{\gamma(s^{-1}r)}\hat{\mu}(\gamma)$$ where $\phi_1^{F}(-,r)$ is interpreted as an element of $C_c(G,C_0(G) \otimes \underline{A}) \subset ({C_0(G)} \otimes \underline{A}) \rtimes G$.
    
    Consider the algebra $C_0(G) \rtimes G$, where $G$ acts on $C_0(G)$ by left-multiplication.  Functoriality induces an action $\kappa$ of $G$ on $C_0(G) \rtimes G$. Let $\kappa'$ be the unique action on $({C_0(G)} \otimes \underline{A}) \rtimes G$ such that the isomorphism $\phi_{C_0(G),A}$ of Lemma \ref{2.5} becomes equivariant with respect to the action $\kappa \otimes \alpha$ on $(C_0(G) \rtimes G) \otimes A$.  We obtain a functor $H_1\colon GC^*\text{Alg}^{nu} \to GC^*\text{Alg}^{nu}$ given by $$(A,\alpha) \mapsto (({C_0(G)} \otimes \underline{A}) \rtimes G),\kappa')$$
    One can check that $\phi_1^{A}$ is $G$-equivariant with the double-dual action on $(A \rtimes G) \rtimes \hat{G}$ and the action $\kappa'$ on $({C_0(G)} \otimes \underline{A}) \rtimes G$. 
    
    As integration commutes with bounded linear operators (in particular, $C^{*}$-algebra morphisms), it follows that the family of isomorphisms $\{\phi_1^{A}\}_{A \in GC^{*}\text{Alg}^{nu}}$ assembles to a natural isomorphism of functors $H \simeq H_1$.

    We now construct a natural isomorphism $H_1 \simeq H'$ as follows:

    In the classical proof, an equivariant isomorphism $$\phi_2\colon({C_0(G)} \rtimes G,\kappa) \to (K(L^2(G)),\text{conjugation})$$ is first constructed. On tensoring by $(A,\alpha)$ and using the isomorphism $\phi_{C_0(G),A}$ of Lemma \ref{2.5}, we obtain an equivariant isomorphism $$\phi_{2}^{A}\colon ({C_0(G)} \otimes \underline{A}) \rtimes G \simeq A \otimes K(L^2(G))$$  The family of isomorphisms $\{\phi_2^{A}\}_{A \in GC^{*}\text{Alg}^{nu}}$ assembles to the desired natural isomorphism.
 
\end{proof}
\begin{kor}\label{3.1.1}
    The functor $(-\rtimes G)\rtimes \hat{G}\colon KK^{G} \to KK^{G}$ is equivalent to the identity functor on $KK^{G}$.
\end{kor}
\begin{proof}
We have a natural (in $A$) zig-zag of morphisms
\[\begin{tikzcd}
	A & {A \otimes \hat K_G} & {A \otimes K_G} & {A \otimes K(L^2(G)) \cong (A \rtimes G) \rtimes \hat G}
	\arrow[from=1-1, to=1-2]
	\arrow[from=1-3, to=1-2]
	\arrow[from=1-4, to=1-3]
\end{tikzcd}\]
Each map above is inverted by $kk^G$ (see Proposition \ref{eqstab}).
    \end{proof}
\begin{kor}\label{eq}
    The functor $-\rtimes G\colon KK^G\to KK^{\hat G}$ is an equivalence of $\infty$-categories with inverse $- \rtimes \hat G$.
\end{kor}
\begin{proof}
    Apply the previous corollary to $G$ and $\hat G$.
\end{proof}

\begin{rem}
    In \cite[Cor. 6.21]{classically}, it is shown, using Takai duality, that the group homomorphisms $J_G$ and $J_{\hat G}$ (see Remark \ref{baajskandalis}) are mutually inverse isomorphisms between $G$-equivariant $KK$-theory groups and $\hat G$-equivariant $KK$-theory groups. In particular, when $G$ is discrete, they set up an equivalence between $KK$-theory for algebras with a $G$-action and $KK$-theory for algebras with a $G$-grading. The corollary above provides a (higher) categorical refinement of this result. 
\end{rem}
Our aim in the remainder of this section is to provide a different formula for the crossed product functor $- \rtimes G\colon KK^{G} \to KK^{\hat G}$ that is closer in form to the duality result of Treumann.

The Hilbert space $L^2(G)$ admits a (unitary) $\hat G$-action defined as follows. Let $f$ be an element of $L^2(G)$, interpreted as a function $f\colon G \to \mathbb C$; then, for a character $\chi$ in $\hat G$, we set $(\chi f)(g)\coloneq \chi(g)f(g)$. 
Conjugation now yields a $\hat G$-action on $K(L^2(G))$. Also, recall the canonical $G$-action on $K(L^2(G))$ by conjugation. 

\begin{lem}\label{3.2}
    The actions of $G$ and $\hat G$ on $K(L^2(G))$ commute and thus yield a $G \times \hat G$-action on $K(L^2(G))$.
\end{lem}
\begin{proof}
    For $\chi \in \hat G$ and $g \in G$, let us denote by $S_\chi$ and $T_g$ the linear operators representing the respective actions of $\chi$ and $g$ on $L^2(G)$. Explicitly, for an element $f$ of $L^2(G)$, we have $S_{\chi}(f)(h) = \chi(h)f(h)$ and $T_g(f)(h) = f(gh)$. Since $S_{\chi}^{-1} T_g$ and $T_g S_{\chi}^{-1}$ are scalar multiples of each other, it follows that $S_{\chi} T_g^{-1}AT_g S_{\chi}^{-1} = T_g^{-1}S_{\chi} AS_{\chi}^{-1} T_g$; so we can define an action of $G \times \hat G$ on $K(L^2(G))$ by $(g,\chi)(A) \coloneq T_gS_{\chi}^{-1}AS_{\chi} T_g^{-1}$.
\end{proof}
 For a $G$-$C^{*}$-algebra $A$, the $C^{*}$-algebra $(A \otimes K(L^2(G))) \rtimes G$ now has two different actions of $\hat G$ - one comes from the action described in Proposition \ref{2.1}, and the other comes from the functoriality of $(A \otimes - ) \rtimes G$ and the $\hat G$-action on $K(L^2(G))$ just described. It will follow from Theorem \ref{3.4} below that these two actions are the same in $KK^{\hat G}$.
 
We denote by $\hypertarget{cast}{\textbf{B}}$ the algebra $K(L^2(G))$ together with this $(G \times \hat G)$-action. We can now define a functor $\label{Takaieq }F\colon GC^{*}\text{Alg}^{nu} \to \hat{G}C^{*}\text{Alg}^{nu}$ by $A \mapsto (A \otimes \textbf{B}) \rtimes G$, where the $\hat G$-action is the one induced by the action of $\hat G$ on $\textbf{B}$. 
\begin{prop}\label{3.3}
    The functor $F\colon GC^{*}\textup{Alg}^{nu} \to \hat{G}C^{*}\textup{Alg}^{nu}$ descends essentially uniquely to a functor $F\colon KK^G \to KK^{\hat G}$.
\end{prop}
\begin{proof}
    We realize $F$ as a composite of three functors -$$i\colon GC^{*}\text{Alg}^{nu} \hookrightarrow (G \times \hat G)C^{*}\text{Alg}^{nu}$$$$-\otimes \textbf{B}\colon (G \times \hat G)C^{*}\text{Alg}^{nu} \to (G \times \hat G)C^{*}\text{Alg}^{nu}$$ $$- \rtimes G\colon (G \times \hat G)C^{*}\text{Alg}^{nu} \to \hat GC^*\text{Alg}^{nu}$$ where, in the third functor, the $G$-action is used to form the crossed product and the $\hat G$-action on the target is induced by functoriality.
    
    One easily checks (using Proposition \ref{eqstab}) that the first functor descends. The second functor descends by the symmetric monoidality of $kk^{G \times \hat G}$. That the third functor descends is seen through an identical\footnote{To see why the steps go through, the main point to note is that the isomorphisms appearing in the proof of Proposition \ref{stabledesc} are $\hat G$-equivariant if $A$ is a $(G \times \hat G)$-$C^{*}$-algebra, and if $V'$ and $V$ are $(G \times \hat G)$-Hilbert spaces. The rest of the argument is trivial.} argument as in the descent of $- \rtimes G\colon GC^{*}\text{Alg}^{nu} \to \hat GC^{*}\text{Alg}^{nu}$ explained in the previous section.
\end{proof}
\begin{theorem}\label{3.4}
There is an isomorphism of functors $F\colon KK^G \to KK^{\hat G}$ and $- \rtimes G: KK^{G} \to KK^{\hat G}$.
\end{theorem}
\begin{proof}
    Write $V$ for $L^2(G)$. Let $\underline{\textbf{B}}$ denote the $G$-$C^{*}$-algebra $K(V)$ with trivial $G$-action. For a $G$-$C^{*}$-algebra $A$, we have the isomorphism of Lemma \ref{2.6} (since translation operators are unitary) $$\psi_{A,\textbf{B}}\colon(A \otimes \textbf{B}) \rtimes G \to (A \otimes \underline{\textbf{B}}) \rtimes G$$
    Identifying $A \otimes K(V)$ with $A$-valued convolution kernels of the form $a(g',g)$, the formula of Lemma \ref{2.6} says that $\psi_{A,\textbf{B}}$ is given by $a_k(g',g) \mapsto a_k (k^{-1}g',g)$.
    
    We give $(A \otimes \textbf{B}) \rtimes G$ the $\hat G$-action of $F(A)$, and define an action on $(A \otimes \underline{\textbf{B}}) \rtimes G$ by $\chi(a_k(g',g)) \coloneq \chi(k(g')^{-1}g)a_k(g',g)$. One can check that $\psi_{A,\textbf{B}}$ is $\hat G$-equivariant with respect to these actions.
    
    We have the natural algebra isomorphism of Lemma \ref{2.5}$$\phi_{A,\underline{\textbf{B}}}\colon (A \rtimes G) \otimes \underline{\textbf{B}} \to (A \otimes \underline{\textbf{B}}) \rtimes G$$ which is given by $f(k) \otimes c(g',g) \mapsto (k \mapsto f(k) \otimes c(g',g))$. We introduce a $\hat G$-action on $(A \rtimes G) \otimes \underline{\textbf{B}}$ by setting $\chi(f(k) \otimes c(g',g)) \coloneq \chi(k)f(k) \otimes \chi((g')^{-1}g)c(g',g)$. One can check that $\phi_{A,\underline{\textbf{B}}}$ is $\hat G$-equivariant with respect to these actions.
    
    We can now make the final comparison with the classical Takai duality functor $ -\rtimes G$. Consider the algebra $A \rtimes G$ with the ``standard" $\hat G$-action, i.e. given by $(\chi f)(k) = \chi(k)f(k)$. As in the proof of Corollary \ref{3.1.1}, we define a zig-zag $\eta _A$ of $\hat G$-equivariant morphisms
\[\begin{tikzcd}
	{A \rtimes G} & {(A \rtimes G) \otimes \hat K_{G}} & {(A \rtimes G) \otimes K_G} & {(A \rtimes G) \otimes K(L^2(G)) = (A \rtimes G) \otimes \underline{\textbf{B}}} & {} & {}
	\arrow[from=1-1, to=1-2]
	\arrow[from=1-3, to=1-2]
	\arrow[from=1-4, to=1-3]
\end{tikzcd}\] It is clear that $\eta_A$ is inverted by $kk^{\hat G}$ (Proposition \ref{eqstab}).
    
    To summarize, one has the following zig-zag of natural equivariant algebra morphisms that become equivalences on descent to $KK^{\hat{G}}$, and thus prove the theorem:
\[\begin{tikzcd}
	{F(A)=(A \otimes \textbf{B}) \rtimes G} && {(A \otimes \underline{\textbf{B}}) \rtimes G} && {(A \rtimes G) \otimes \underline{\textbf{B}}} & {A \rtimes G}
	\arrow["{\psi_{A,\textbf{B}}}", from=1-1, to=1-3]
	\arrow["{\phi_{A,\underline{\textbf{B}}}}"', from=1-5, to=1-3]
	\arrow["{\eta_A}"', from=1-6, to=1-5]
\end{tikzcd}\]
\end{proof}
\section{Treumann Duality}\label{four}
We begin with some preliminaries.

For a ring spectrum $R$, we write $\text{LMod}(R)$ for the presentable stable $\infty$-category of left modules\footnote{When $R$ is commutative, we shall simply write $\text{Mod}(R)$.} over $R$. We will denote by $\text{Perf}(R) \subset \text{LMod}(R)$ the full stable subcategory of perfect modules. Given a space $X$, we set $R[X] \coloneq R \wedge \Sigma_{+}^{\infty} X$. If $X$ is a group in spaces, $R[X]$ acquires the structure of a ring spectrum from the ring structure on $R$ and the group structure on $X$. If $R$ and $G$ are commutative as a ring spectrum and group, respectively, $R[X]$ is a commutative ring spectrum.

Given ring spectra $R$ and $S$, we remark that the colimit-preserving functors $\text{LMod}(R) \to \text{LMod}(S)$ correspond precisely to $(R,S)$-bimodules; see \cite [Prop. 7.1.2.4]{ha} for a precise statement.
\begin{prop}\label{long}
    Let $X$ be a connected space. There is an equivalence of $\infty$-categories over $\textup{LMod}(R)$ $$\textup{LMod}(R[\Omega X]) \simeq \textup{Fun}(X,\textup{LMod}(R))$$
\end{prop}
\begin{proof}
    Let $\underline{R}: * \to \textup{LMod}(R)$ be the functor that picks out the object $R \in \textup{LMod}(R)$, and let $i_{!}\underline{R}: X \to \text{LMod}(R)$ denote the functor obtained by left Kan extending $\underline{R}$ along a basepoint inclusion $i:* \to X$. One readily checks that $i_!\underline{R}$ generates $\text{Fun}(X,\text{LMod}(R))$ in the sense of the variant (\cite[Thm. 7.1.2.1]{ha}) of the Schwede-Shipley theorem. Further, an easy calculation\footnote{As monoids, $\Omega X \simeq \text{Nat}(i,i)$; so one obtains a monoid map $\Omega X \to \text{End}(i_!\underline{R})$ which induces the desired identification.} shows that the endomorphism algebra of $i_{!}\underline{R}$ can be identified with $R[\Omega X]$ as algebras, and thus the desired equivalence follows from (\cite[Rem. 7.1.2.3]{ha}). To see that the equivalence is one over $\text{LMod}(R)$, one only has to observe that both the forgetful functor $\text{LMod}(R[\Omega X]) \to \text{LMod}(R)$ and the composite of the equivalence with the evaluation functor $\text{LMod}(R[\Omega X]) \cong \text{Fun}(X,\text{LMod}(R)) \to \text{LMod}(R)$ preserve colimits, and take the compact generator $R[G]$ of $\text{LMod}(R[G])$ to $R[G]$ regarded as an object of $\text{LMod}(R)$.
\end{proof} 
If $G$ is a group, the equivalence above specializes to $\text{LMod}(R[G]) \cong \text{Fun}(BG,\text{LMod}(R))$.

We write $\text{LMod}(R[G])^{ft}$ for the full subcategory of $\text{LMod}(R[G])$ generated by those modules whose homotopy groups are finitely generated over $\pi_0(R)$. We refer to the objects of $\text{LMod}(R[G])^{ft}$ as modules of finite type over $R[G]$.

Fix a prime $p$ for the remainder of this section, and let $KU_p$ denote the $p$-completion of $KU$. \begin{prop}\label{4.1}
\begin{enumerate}
\leavevmode
    \item For a finite group $G$, the equivalence of Proposition \ref{long} induces an equivalence $$\textup{LMod}(KU_p[G])^{ft} \simeq \textup{Fun}(BG, \textup{Perf}(KU_p))$$ 
    \item Let $G$ be a finite group and $M \in \textup{Fun}(BG,\textup{Perf}(KU_p))$. Then the $p$-completed homotopy orbits $(M_{hG})_p \in \textup{Perf}(KU_p)$.
    \item For finite abelian groups $G_1$ and $G_2$, and a $KU_p[G_1 \times G_2]$-bimodule $M$ of finite type, the functor
\[\begin{tikzcd}
	{\textup{Mod}(KU_p[G_1])} && {\textup{Mod}(KU_p[G_2])} 
	\arrow["{(-\otimes_{KU_p[G_1]} M)_p}", from=1-1, to=1-3]
\end{tikzcd}\] sends modules of finite type to modules of finite type.
\end{enumerate}
\end{prop}
\begin{proof}
\leavevmode
    \begin{enumerate}
        \item Since the equivalence of Proposition \ref{long} preserves the underlying $KU_p$-module, it suffices to prove that over $KU_p$, the finite type modules are precisely the perfect modules. One readily sees that every perfect module is of finite type.
 
        We will now show that $\text{Mod}(KU_p)^{ft} \subset \textup{Mod}(KU_p)^{\textup{perf}}$. Suppose $M$ is a $KU_p$-module with finitely generated homotopy groups over $\mathbb Z_p$. Since $\pi_*(KU_p)$ has global dimension $1$ as a graded ring, we can find a resolution $0 \to F_1 \to F_0 \to \pi_*(M) \to 0$ where $F_1$ and $F_0$ are finitely generated free modules over $\pi_*(KU_p)$. This lifts to an exact square in $\text{Mod}(KU_p)$
\[\begin{tikzcd}
	{M_1} && {M_0} \\
	\\
	0 && M
	\arrow[from=1-3, to=3-3]
	\arrow[from=1-1, to=1-3]
	\arrow[from=1-1, to=3-1]
	\arrow[from=3-1, to=3-3]
\end{tikzcd}\]
        where both $M_1$ and $M_0$ are finite direct sums of $KU_p$ and $\Sigma KU_p$. This finishes the proof.
\item Let $P$ be a $p$-Sylow subgroup of $G$. As in the proof of \cite[Cor. A.10]{treumann2015representations}, we note that $M$ is a retract of $M \otimes KU_p[G/P]$ in $\text{LMod}(KU_p[G])$, so $M_{hG}$ is a retract of $M_{hP}$. Therefore we reduce to the case that $G$ is a $p$-group. For a normal subgroup $N \subset G$, we have that $M_{hG} \simeq (M_{hN})_{h(G/N)}$; as the center of a $p$-group is non-trivial, we reduce to the case that $G = C_p$. By \cite[Thm. A.11]{treumann2015representations}, it suffices to check the assertion when $M = KU_p$ with the trivial $G$-action and when $M = KU_p[G]$ with $G$ acting by permutations. In the latter case, $M_{hG}$ is simply $KU_p$. In the former case, by \cite[Thm. 0.0.1]{lurie}, $(M_{hG})_p$ identifies with $M^{hG}$. By the case $X = *$ of the strong version of the Atiyah-Segal completion theorem \cite[Prop. 4.2]{completion} and the observation that completing at the augmentation ideal is the same as $p$-completion for $G = C_p$, we deduce\footnote{The cited reference states that the fixed points of $KU$ with the trivial $G$-action is $\oplus _G KU_p$, but $p$-completion commutes with fixed points.} that $M^{hG} \cong \oplus_G KU_p$, which is clearly perfect.  This concludes the proof.
\item If $N \in \text{Mod}(KU_p)[G_1]^{ft}$, the K\"unneth theorem yields that $N \otimes_{KU_p} M \in \text{Mod}(KU_p)[G_2]^{ft}$. By the previous part we now conclude that $$ ((N \otimes_{KU_p} M)_{hG_1})_p \simeq (N \otimes_{KU_p[G_1]} M)_p $$ is perfect.
\end{enumerate}
\end{proof}
We now briefly recall the setup for topological $K$-theory of $C^{*}$-categories, as described in \cite{bunke2021additive}.

We begin with a recollection of some basic definitions as presented in Section 2 of \textit{loc.cit.}. A possibly non-unital $\mathbb C$-linear $*$-category is a possibly non-unital category enriched over the category of $\mathbb C$-vector spaces together with an involution $*$ that fixes the objects, reverses the direction of the morphisms, and acts complex anti-linearly on the morphism spaces. We have a category $^{*}\text{Cat}^{nu}_{\mathbb C}$ with objects the possibly non-unital $\mathbb C$-linear $*$-categories, and morphisms the $\mathbb C$-linear functors of the underlying non-unital categories that additionally preserve the involution $*$. For a Hilbert space $H$, the $C^{*}$-algebra $B(H)$ can be regarded as a $\mathbb C$-linear $*$-category with one object. We now define the maximal norm $||-||_{\text{max}}$ on the morphism spaces of a possibly non-unital $\mathbb C$-linear $*$-category $\mathcal{C}$. For a morphism $f$ of $\mathcal{C}$, we set $$||f||_{\text{max}} \coloneq \text{sup}_{\rho\colon \mathcal{C} \to B(H)}||\rho(f)||$$where the supremum is taken over all morphisms $\rho$ in $^{*}\text{Cat}^{nu}_{\mathbb C}$, and $B(H)$ is given the operator norm. An object of $^{*}\text{Cat}^{nu}_{\mathbb C}$ is called a $C^{*}$-category if all its morphism spaces are complete with respect to $||-||_{\text{max}}$ \footnote{This includes the assumption that $||-||_{\text{max}}$ is indeed a norm on the morphism spaces.}. We denote by $C^{*}\text{Cat}^{nu}$ the full subcategory of $^{*}\text{Cat}^{nu}_{\mathbb C}$ spanned by the $C^{*}$-categories \footnote{Henceforth, whenever we speak of functors between $C^{*}$-categories, we always refer to morphisms in $C^{*}\text{Cat}^{nu}$.}. A $C^{*}$-algebra can be regarded as a $C^{*}$-category with a single object. This provides an inclusion functor $i\colon C^{*}\text{Alg}^{nu} \hookrightarrow C^{*}\text{Cat}^{nu}$, which admits a left adjoint $A^{f}\colon C^{*}\text{Cat}^{nu} \to C^{*}\text{Alg}^{nu}$. We now define the $K$-theory functor $K^{\textbf{C}}$ as the composite 
\[\begin{tikzcd}
	{K^{\textbf{C}}\colon C^{*}\text{Cat}^{nu}} && {C^{*}\text{Alg}^{nu}} && {\text{Mod}(KU)}
	\arrow["{A^{f}}", from=1-1, to=1-3]
	\arrow["{K}", from=1-3, to=1-5]
\end{tikzcd}\]
The composite $A^f \circ i$ is isomorphic to the identity functor, so for a $C^{*}$-algebra $A$ one has $K^{\textbf{C}}(i(A)) \simeq K(A)$.

Let  $C^{*}\text{Cat}$ be the subcategory of $C^{*}\text{Cat}^{nu}$ comprising of the unital $C^{*}$-categories and unital $C^{*}$-functors. We abuse notation and write $K^{\textbf{C}}$ also for the composite 
\[\begin{tikzcd}
	{C^{*}\text{Cat}} && {C^{*}\text{Cat}^{nu}} && {\text{Mod}(KU)}
	\arrow[hook, from=1-1, to=1-3]
	\arrow["{K^{\textbf{C}}}", from=1-3, to=1-5]
\end{tikzcd}\]
We next describe classes of morphisms in $C^{*}\text{Cat}$ that are inverted by $K^{\textbf{C}}$. This helps felicitate computations as the functor $A^{f}$ has a complicated description.

We first describe the notion of a unitary equivalence between unital $C^{*}$-categories (\cite[Sec. 8.5.4]{bunke2020homotopy}). Let $\textbf{C}$ and $\textbf{D}$ be unital $C^{*}$-categories, and $f,g\colon \textbf{C} \to \textbf{D}$ be unital functors. We say that $f$ and $g$ are unitarily isomorphic if there exists a natural transformation $u\colon f \to g$ such that $u(c)\colon f(c) \to g(c)$ is unitary for each object $c$ of $\textbf{C}$. A unital functor $f\colon \textbf{C} \to \textbf{D}$ between unital $C^{*}$-categories is a unitary equivalence if there exists a functor $g\colon \textbf{D} \to \textbf{C}$ such that $f \circ g$ and $g \circ f$ are isomorphic to the respective identity functors on $\textbf{D}$ and $\textbf{C}$. If $f$ is a unitary equivalence, then $K^{\textbf{C}}(f)$ is an equivalence (\cite[Cor. 8.63]{bunke2020homotopy}).

Thus, if $W$ denotes the class of unitary equivalences in $C^{*}\text{Cat}$, one has the following factorization: 
\[\begin{tikzcd}
	{C^{*}\text{Cat}} &&& {\text{Mod}(KU)} \\
	& {C^{*}\text{Cat}[W^{-1}]}
	\arrow["{K^{\textbf{C}}}", from=1-1, to=1-4]
	\arrow["{L_W}"', from=1-1, to=2-2]
	\arrow["K^{\mathcal{C}}",from=2-2, to=1-4]
\end{tikzcd}\]
where $L_W$ is the localization functor. 

 In \cite{ambrcat}, the construction of a maximal tensor product $-\otimes-$, of two unital $C^{*}$-categories is described, and it is shown that this endows $C^{*}\text{Cat}$ with a symmetric monoidal structure such that the inclusion $C^{*}\text{Alg} \hookrightarrow C^{*}\text{Cat}$ is symmetric monoidal. In \cite{bunke2023stable}, this was upgraded - a maximal tensor product $-\otimes-$ is constructed on $C^{*}\text{Cat}^{nu}$ which gives it a symmetric monoidal structure such that the inclusions $C^{*}\text{Cat} \hookrightarrow C^{*}\text{Cat}^{nu}$ and $i\colon C^{*}\text{Alg}^{nu} \hookrightarrow C^{*}\text{Cat}^{nu}$ are both symmetric monoidal. 

\begin{prop}
\leavevmode
\begin{enumerate} 
    \item The localization functor $L_W$ admits a symmetric monoidal refinement.
    \item The functor $K^{\mathcal{C}}$ admits a lax symmetric monoidal refinement for this symmetric monoidal structure on $C^{*}\textup{Cat}[W^{-1}]$.
    \end{enumerate}
\end{prop}
\begin{proof}
\leavevmode
    \begin{enumerate}
        \item This follows from the fact that for any unitary equivalence $f\colon \textbf{C} \to \textbf{C}'$ and any object $\textbf{D}$ of $C^{*}\text{Cat}$, $f \otimes \text{id}_{\textbf{D}}$ and $\text{id}_{\textbf{D}} \otimes f$ are unitary equivalences.
        \item Write $kk^{\textbf{C}}$ for the composition 
\[\begin{tikzcd}
	{C^{*}\text{Cat}^{nu}} && {C^{*}\text{Alg}^{nu}} && KK
	\arrow["{A^{f}}", from=1-1, to=1-3]
	\arrow["kk", from=1-3, to=1-5]
\end{tikzcd}\]
 \cite[Prop 7.9]{bunke2023stable} states that $kk^{\textbf{C}}$ is symmetric monoidal. 

Further, \cite[Cor. 6.14]{bunke2023stable} states that $kk^{\textbf{C}}$ inverts unitary equivalences between unital $C^{*}$-categories, so item 1 gives us a symmetric monoidal functor $kk^{\mathcal{C}}\colon C^{*}\textup{Cat}[W^{-1}] \to KK$. We thus have the following commutative diagram:
\[\begin{tikzcd}
	{C^{*}\text{Cat}} & {C^{*}\text{Cat}^{nu}} &&& KK && {\text{Mod}(KU)} \\
	\\
	&& {C^{*}\text{Cat}[W^{-1}]}
	\arrow[hook, from=1-1, to=1-2]
	\arrow["{kk^{\textbf C}}", from=1-2, to=1-5]
	\arrow["{KK(\mathbb C,-)}", from=1-5, to=1-7]
	\arrow["{L_W}"', from=1-1, to=3-3]
	\arrow["{kk^{\mathcal{C}}}", from=3-3, to=1-5]
	\arrow["{K^{\mathcal{C}}}"', from=3-3, to=1-7]
	\arrow["{K^{\textbf{C}}}", curve={height=-30pt}, from=1-2, to=1-7]
\end{tikzcd}\]
As $K^{\mathcal{C}}$ is equivalent to the composite of the symmetric monoidal functor $kk^{\mathcal{C}}$ followed by the lax symmetric monoidal functor $KK(\mathbb C, -)$, the statement follows.
    \end{enumerate}
\end{proof}

We now turn to Morita equivalences. We refer the reader to  \cite[Def. 5.5]{bunke2021additive} and \cite[Def. 2.19]{bunke2021additive} for the definitions of additive and idempotent complete unital $C^{*}$-categories, respectively. We say that a morphism $f\colon \mathcal{C} \to \mathcal{D}$ in $C^{*}\text{Cat}$ presents $\mathcal{D}$ as the additive completion of $\mathcal{C}$ if $\mathcal{D}$ is additive, $f$ is fully faithful, and every object of $\mathcal{D}$ is unitarily isomorphic to a finite orthogonal sum of objects in the image of $f$ (Definition 16.1 of \textit{loc.cit.}). Similarly, we say that a morphism $f\colon \mathcal{C} \to \mathcal{D}$ in $C^{*}\text{Cat}$ presents $\mathcal{D}$ as the idempotent completion of $\mathcal{C}$ if $\mathcal{D}$ is idempotent complete, $f$ is fully faithful, and for each object $d$ of $\mathcal{D}$, there is an isometry $d \to d'$, where $d'$ is an object in the image of $f$ (Definition 16.4 of \textit{loc.cit.}).

Let $C^{*}\text{Cat}^{\text{Idem}}_{\oplus}$ denote the full subcategory of $C^{*}\text{Cat}$ spanned by the additive and idempotent complete $C^{*}$-categories. There is a functor (see 16.1 of \textit{loc.cit}) $(-)^{\#}\colon C^{*}\text{Cat} \to C^{*}\text{Cat}^{\text{Idem}}_{\oplus}$ together with a natural transformation $u\colon\text{id} \to (-)^{\#}$ such that for each object $\mathcal{C}$ of $C^{*}\text{Cat}$, the map $\mathcal{C} \xrightarrow{u_{\mathcal{C}}} \mathcal{C}^{\#}$ presents $\mathcal{C}^{\#}$ as the additive and idempotent completion of $\mathcal{C}$. We define a Morita equivalence to be a morphism in $C^{*}\text{Cat}$ that is sent to a unitary equivalence by $(-)^{\#}$ (Definition 16.7 of \textit{loc.cit.}). Morita equivalences are inverted by $K^{\textbf{C}}$ (Theorem 16.18 of \textit{loc.cit.}).

Given a unital $C^{*}$-algebra $A$, one can consider the category of finitely generated projective Hilbert $A$-modules, which we denote by $\text{Hilb}(A)^{\text{fg,proj}}$. This acquires the structure of a unital $C^{*}$-category in the obvious way. As argued in Example 16.9 of \textit{loc.cit.}, the inclusion $i(A) \to \text{Hilb}(A)^{\text{fg,proj}}$ is a Morita equivalence, so it induces an equivalence $K^{\textbf{C}}(\text{Hilb}(A)^{\text{fg,proj}}) \simeq K(A)$. 

\begin{rem}\label{lone}
Let $A \coloneq \text{Hilb}(\mathbb C)^{\text{fg,proj}}$. For all $B \in C^{*}\textup{Cat}[W^{-1}]$, the lax monoidal structure map $$K^{\mathcal{C}}(L_W(A)) \otimes_{KU} K^{\mathcal{C}}(B) \to K^{\mathcal{C}}(L_W(A) \otimes B)$$ is an equivalence because of the Morita invariance of $K^{\textbf{C}}$.
\end{rem}
In \cite{bunke2019homotopy}, the existence of a simplicial model structure on $C^{*}\text{Cat}$ is shown, in which the weak equivalences are precisely the unitary equivalences and every object is both cofibrant and fibrant. Therefore, the localization $C^{*}\text{Cat}[W^{-1}]$ can be identified with the homotopy coherent nerve $N_h(C^{*}\text{Cat})$. From the explicit description of the simplicial structure in Definition 6.19 of \textit{loc.cit.}, one sees that $N_h(C^{*}\text{Cat})$ is the nerve of a $2$-category.

 We now give an explicit description of $N_h(C^{*}\text{Cat})$. Let $A \in C^{*}\text{Cat}$, and let $\Pi(\Delta^{1})$ denote the fundamental groupoid of the simplicial set $\Delta^{1}$. We describe the unital $C^{*}$-category $A \otimes \Pi(\Delta^{1})$ that serves as the cylinder object on $A$. Its objects are $\text{Ob}(A) \times \{0,1\}$. Its morphisms are given by, for $j,j' \in \{0,1\}$ and $a,a' \in \text{Ob}(A)$, $$\text{Hom}_{A \otimes \Pi(\Delta^{1})}((a,j),(a',j')) = \text{Hom}_{A}(a,a')$$ with the obvious definitions of composition and $*$. We have unital functors $i_0, i_1\colon A \to A \otimes \Pi(\Delta_1)$ given by $a \mapsto (a,0)$ and $a \mapsto (a,1)$, respectively.

  The data of a $2$-simplex in $N_h(C^{*}\text{Cat})$ is that of unital $C^{*}$-categories $A, B, C$, unital functors $f_0\colon A \to B$, $f_1\colon B \to C$, $f_2\colon A \to C$, and a unital functor $h\colon A \otimes \Pi(\Delta^{1}) \to C$ such that $h \circ i_0 = f_1 \circ f_0$ and $h \circ i_1 = f_2$ (see \cite[\href{https://kerodon.net/tag/00M5}{Tag 00M5}]{kerodon}).

 In what follows, we denote by $U(1)^{\delta}$ the unitary group of $\mathbb C$ endowed with the discrete topology. We define a functor $\Phi\colon BBU(1)^{\delta} \to N_h(C^{*}\text{Cat})$ by specifying the data of the object $\text{Mod}(\mathbb C) \in C^{*}\text{Cat}$, the identity functor $\text{Mod}(\mathbb C) \to \text{Mod}(\mathbb C)$, and the action of elements of $U(1)$ as automorphisms of this identity functor by scalar multiplication. 
 
Let $G$ be a finite abelian $p$-group. Suppose one has a map $B(G \times \hat G) \to BBU(1)^{\delta}$; on post-composing with $\Phi$, we get a functor $B(G \times \hat G) \to N_h(C^{*}\text{Cat})$. This gives an action of $G \times \hat G$ on the category $\text{Mod}(\mathbb C)$ regarded as an object of $N_h(C^{*}\text{Cat})$. Further post-composing by $ K^{\mathcal{C}}$ gives a $G \times \hat G$-action on $KU$ regarded as an object of $\text{Mod}(KU)$, and therefore determines a $KU[G \times \hat G]$-module. 

The equivalence classes of maps of spaces from $B(G \times \hat G) \to BBU(1)^{\delta}$ can be identified with the cohomology group $H^2(B(G \times \hat G),U(1)^{\delta})$ (where we mean that $U(1)^{\delta}$ is considered with the trivial action of $G \times \hat G$).

Consider the $U(1)$-valued $2$-cocycle $E\coloneq ((g,\chi),(g',\chi')) \mapsto \chi'(g)$. In the manner explained above, this determines a $KU[G \times \hat G]$-module $\hypertarget{cat}{\textbf{M'}}$. On $p$-completing $\label{M'_E}\textbf{M'}$, we get a $KU_p[G \times \hat G]$-module which we shall denote by $\hypertarget{M_E}{\textbf{M}}$.

We can now formulate Theorem \ref{secondo} more precisely.
\begin{theorem}[\textbf{Treumann Duality}]\label{4.3}
    The functor  $\textup{Mod}(KU_p[G])^{ft} \to  \textup{Mod}(KU_p[\hat G])^{ft}$ given by $N \mapsto (N \otimes_{KU_p[G]}\textup{\textbf{M}})_p$ is an equivalence.
\end{theorem}
    Note that the definition of the functor above makes sense in light of item 3 of Proposition \ref{4.1}. 
\begin{proof}
    We will first sketch why our bimodule $\textbf{M}$ coincides with the bimodule $M_E$ defined in \cite[Sec. 3.6]{treumann2015representations} to justify that our statement is indeed Treumann's result. We will then outline his proof.

    We first observe that the assignment $A \mapsto \text{Hilb}(A)^{\text{fg,proj}}$ actually determines a functor $C^{*}\text{Alg} \to \text{CMon}(C^{*}\text{Cat})$. One easily sees that the action of $G \times \hat{G}$ on $\text{Mod}(\mathbb C)$ regarded as an object of $N_h(C^{*}\text{Cat})$ refines to an action on $\text{Mod}(\mathbb C)$ regarded as an object of $\text{CMon}(N_h(C^{*}\text{Cat}))$. We also observe that taking the core of the underlying topologically enriched category yields a functor $\text{CMon}(C^{*}\text{Cat}) \to \text{CMon}(\text{Spc})$. As this process inverts unitary equivalences, it induces a functor $\text{CMon}(N_h(C^{*}\text{Cat})) \to \text{CMon}(\text{Spc})$.

    In the following diagram, we claim that all composites from $C^{*}\text{Alg}$ to $\text{CGrp}(\text{Spc})$ are naturally isomorphic to the connective $K$-theory functor on $C^{*}$-algebras:
\[\begin{tikzcd}
	{C^{*}\text{Alg}} && {\text{CMon}(C^{*}\text{Cat})} & {\text{CMon}(N_h(C^{*}\text{Cat}))} & {N_h(C^{*}\text{Cat})} & {\text{Sp}} \\
	{} && {\text{CMon}(\text{Spc})} & {\text{CGrp}(\text{Spc})}
	\arrow["{\text{Hilb}(-)^{\text{fg,proj}}}", from=1-1, to=1-3]
	\arrow[from=1-3, to=1-4]
	\arrow["{(-)^{\simeq}}"', from=1-3, to=2-3]
	\arrow[from=1-4, to=1-5]
	\arrow["{(-)^{\simeq}}"{description}, from=1-4, to=2-3]
	\arrow["{K^{\mathcal{C}}}", from=1-5, to=1-6]
	\arrow["{\tau_{\geq 0}}", from=1-6, to=2-4]
	\arrow["{(-)^{\text{grp}}}"', from=2-3, to=2-4]
\end{tikzcd}\]
The case of the right-down peripheral composite is handled by the Morita invariance of $K^{\mathcal{C}}$, while that of the other peripheral composite is just the fact that the definition of $K$-theory we use coincides with the classical one. Treumann's bimodule $M_E$ is defined via an action of $G \times \hat G$ on $\text{Mod}(\mathbb C) \in \text{CMon}(\text{Spc})$, and our bimodule $\textbf{M}$ via an action of $G \times \hat G$ on $\text{Mod}(\mathbb C) \in \text{CMon}(N_h(C^{*}\text{Cat}))$. One immediately sees that these actions are compatible, so it follows that $\textbf{M} \simeq M_E$.

    We now outline Treumann's proof in \cite[Sec. 3.6]{treumann2015representations} (adapted to $\textbf{M}$).
    
    Regarding $\textbf{M}$ as a $KU_p[\hat G \times G]$-module yields a functor  $\textup{Mod}(KU_p[\hat G])^{ft} \to  \textup{Mod}(KU_p[G])^{ft}$. The cited argument shows that this functor is the inverse of the functor of the theorem, by showing that $(\textbf{M}_{\hat G} \otimes_{KU_p[\hat G]} \textbf{M}_{G})_p$  and $(\textbf{M}_{G} \otimes_{KU_p[G]} \textbf{M}_{\hat G})_p$ are isomorphic to the diagonal $(G,G)$ and $(\hat G,\hat G)$-bimodules over $KU_p$, respectively (here, we write $\textbf{M}_{G}$ for $\textbf{M}$ regarded as a  $KU_p[\hat G \times G]$-module and $\textbf{M}_{\hat G}$ for $\textbf{M}$ regarded as a  $KU_p[G \times \hat G]$-module). This involves computations of group actions that are performed on the level of the module category. These go through verbatim in our situation by Remark \ref{lone} and the fact that $\text{Mod}(\mathbb C)$ is idempotent with respect to the symmetric monoidal structure of the maximal tensor product.
\end{proof}
\section{A Comparison of Takai Duality with Treumann Duality}\label{five}
In this section, we compare Takai duality with Treumann duality. 

Fix a prime number $p$ and a finite abelian $p$-group $G$ for the rest of this section. Let $(-)_p$ denote the $p$-completion functor.

We first define a functor $r_G\colon GC^*\text{Alg}^{nu} \to \text{Mod}(KU_p[G])$ as the composite 
\[\begin{tikzcd}
	{GC^*\text{Alg}^{nu}} && {\text{Fun}(BG,C^*\text{Alg}^{nu})} && {\text{Fun}(BG,KK)} \\
	\\
	{\text{Mod}(KU_p[G])} && {\text{Mod}(KU[G])} && {\text{Fun}(BG,\text{Mod}(KU))}
	\arrow["\sim", from=1-1, to=1-3]
	\arrow["{-\circ kk}", from=1-3, to=1-5]
	\arrow["{-\circ K(-)}", from=1-5, to=3-5]
	\arrow["\sim"', from=3-5, to=3-3]
	\arrow["{(-)_p}"', from=3-3, to=3-1]
\end{tikzcd}\]
One can check that the upper horizontal composite is homotopy invariant, $G$-stable, semi-exact, and $s$-finitary, essentially because $kk$ has those properties for $G = \{e\}$. Thus, one has an essentially unique descended functor $\widehat{\text{Res}}_G\colon KK^G \to \text{Fun}(BG,KK)$ such that the following diagram commutes:
\[\begin{tikzcd}
	{GC^*\text{Alg}^{nu}} && {\text{Fun}(BG,C^*\text{Alg}^{nu})} \\
	\\
	{KK^G} && {\text{Fun}(BG,KK)}
	\arrow["\sim", from=1-1, to=1-3]
	\arrow["{-\circ kk}", from=1-3, to=3-3]
	\arrow["{kk^G}"', from=1-1, to=3-1]
	\arrow["{\widehat{\text{Res}}_G}"', from=3-1, to=3-3]
\end{tikzcd}\]

 It follows that $r_G$ also descends essentially uniquely to a functor $KK^G \to \text{Mod}(KU_p[G])$, which we shall also denote by $r_G$. 

We now restrict our attention to those objects of $KK^G$ that map to finite-type modules. Let $KK^{G,ft}$ denote the full subcategory of $KK^G$ spanned by objects $A$ such that $r_G(A) \in \text{Mod}(KU_p[G])^{ft}$.
\begin{prop}
    $KK^{G,ft}$ is a thick subcategory of $KK^G$.
\end{prop}
\begin{proof}
    Clearly, $0 \in KK^{G,ft}$. It is also obvious that $KK^{G,ft}$ is closed under retracts.
    
    We observe that $r_G$ is exact. As proved in item 1 of Proposition \ref{4.1}, $\text{Mod}(KU_p[G])^{ft}$ is stable. Therefore, $KK^{G,ft}$ is stable and closed under fibers and cofibers. 
\end{proof}
We now recall the construction of an assembly map $$\textup{colim}_{BG}K(\widehat{\textup{Res}}_G(-)) \to K(- \rtimes G)$$ of functors $KK^{G} \to \textup{Mod}(KU)$ as described in the proof of \cite[Prop. 2.7]{bunke2023ktheory}, which we shall crucially use in what follows.

Bunke first constructs a natural transformation of functors $KK^{G} \to \text{Fun}(BG,KK)$ $$c'_G\colon\widehat{\text{Res}}_G(-) \to \text{triv}^{G}(- \rtimes G)$$where $\text{triv}^{G}\colon KK \to \text{Fun}(BG,KK)$ equips\footnote{We also use the notation $\text{triv}^{G}$ to indicate the functor $KK^{G} \to \text{Fun}(BG,KK^{G})$ that equips an object with the trivial $G$-action. It will be clear from the context as to which functor we have in mind.} an object with the trivial $G$-action. The map $c'_G$ is obtained by the following identifications (see \textit{loc.cit.} for the details):\begin{align*}
\widehat{\text{Res}}_G(-) &\simeq (\text{triv}^{G}(-) \otimes C_0(G)) \rtimes G\\ &\simeq( \text{triv}^{G}(-) \otimes *_{G}\mathbb C) \rtimes G\\& \xrightarrow{G \to *} \text{triv}^{G}(- \rtimes G)
\end{align*}

Its adjoint is a natural transformation $$c_G\colon \text{colim}_{BG}\widehat{\text{Res}}_G(-) \to - \rtimes G$$ The desired assembly map is obtained by applying the colimit preserving functor $K(-)$ to $c_G$.
\begin{theorem}\label{4.5}
    Let $G$ be a finite abelian $p$-group. The assembly map $$\textup{colim}_{BG}K(\widehat{\textup{Res}}_G(-)) \to K(- \rtimes G)$$ constructed above becomes an equivalence after $p$-completion. 
\end{theorem}
\begin{proof}
    We begin with a lemma.
    \begin{lem}\label{prod}
        Let $H_1$ and $H_2$ be finite abelian groups, and set $G \coloneq H_1 \times H_2$. There exists an essentially unique functor $F\colon KK^{G} \to KK^{H_1}$ such that the following diagram is commutative:
\[\begin{tikzcd}
	{KK^{G}} && {\textup{Fun}(BH_1,KK^{H_2})} && {\textup{Fun}(BH_1,KK)} \\
	\\
	KK &&&& {KK^{H_1}}
	\arrow["\widehat{\textup{Res}}_{G,H_2}",from=1-1, to=1-3]
	\arrow["{-\circ(\rtimes H_2)}", from=1-3, to=1-5]
	\arrow["\widehat{\textup{Res}}_{H_1}",from=3-5, to=1-5]
	\arrow["{\rtimes H_1}", from=3-5, to=3-1]
	\arrow["{\rtimes G}"', from=1-1, to=3-1]
	\arrow["F", from=1-1, to=3-5]
\end{tikzcd}\]
    \end{lem}
    \begin{proof}
        First, one checks that the following diagram is commutative:
\[\begin{tikzcd}
	{GC^{*}\text{Alg}^{nu}} && {\text{Fun}(BH_1,H_2C^{*}\text{Alg}^{nu})} && {\text{Fun}(BH_1,C^{*}\text{Alg}^{nu})} \\
	\\
	{C^{*}\text{Alg}^{nu}} &&&& {H_1C^{*}\text{Alg}^{nu}}
	\arrow["\sim", from=1-1, to=1-3]
	\arrow["{-\circ(\rtimes H_2)}", from=1-3, to=1-5]
	\arrow["\sim"', from=3-5, to=1-5]
	\arrow["{\rtimes H_1}", from=3-5, to=3-1]
	\arrow["{\rtimes G}"', from=1-1, to=3-1]
\end{tikzcd}\]
It is easy to verify that each functor in this diagram can be descended to a functor between the appropriate equivariant $KK$ categories. Using the uniqueness properties of the descended functors several times, the lemma follows.
    \end{proof}

We write $G$ as a direct product of $n$ cyclic groups, each of prime power order. We will proceed by induction on $n$.

If $n = 1$, the result follows from \cite[Thm. 5.7]{bunke2023ktheory}.

So assume $G = H_1 \times H_2$ and that the theorem holds for $H_1$ and $H_2$. Let $F$ be the functor of Lemma \ref{prod}.

To simplify the proof, we introduce the following notation and record some obvious relations:
\begin{itemize}
    \item (Colimits) $L_0 \coloneq \text{colim}_{BG}\colon \text{Fun}(BG,KK) \to KK$, $L_1\coloneq \text{colim}_{BH_1}\colon \text{Fun}(BH_1,KK) \to KK$, $L_2 \coloneq \text{colim}_{BH_2}\colon \text{Fun}(BG,KK) \to \text{Fun}(BH_1,KK)$; $L_0 \simeq L_1 \circ L_2$.
    \item (Restrictions) $M_0\coloneq \widehat{\text{Res}}_G\colon KK^{G} \to \text{Fun}(BG,KK)$, $M_1 \coloneq \widehat{\text{Res}}_{G,H_2}\colon KK^{G} \to \text{Fun}(BH_1,KK^{H_2})$, $M_2 \coloneq \widehat{\text{Res}}_{H_2}\colon \text{Fun}(BH_1,KK^{H_2}) \to \text{Fun}(BG,KK)$; $M_0 \simeq M_2 \circ M_1$.
    \item (Crossed Products) $N_0 \coloneq -\rtimes G\colon KK^{G} \to KK$, $N_1\coloneq -\rtimes H_1\colon KK^{H_1} \to KK$, $N_2\coloneq -\rtimes H_2\colon KK^{G} \to KK^{H_1}$, $N'_2\coloneq -\rtimes H_2\colon \text{Fun}(BH_1,KK^{H_2}) \to \text{Fun}(BH_1,KK)$; $N_0 \simeq N_1 \circ N_2$.
\end{itemize}

Note that $c_{H_2}$ induces a map $L_2M_2(-) \to N'_2(-)$. We will abuse notation and also refer to this as $c_{H_2}$.

Abusing notation slightly, we observe that $c'_{G}$ is equivalent to the composite \begin{align*}
\widehat{\text{Res}}_G(-) &\simeq (\text{triv}^{G}(-) \otimes C_0(G)) \rtimes G\\ &\simeq (\text{triv}^{G}(-) \otimes C_0(H_1) \otimes C_0(H_2)) \rtimes G\\ &\simeq (\text{triv}^{H_1}((\text{triv}^{H_2}(-) \otimes C_0(H_2)) \rtimes H_2) \otimes C_0(H_1)) \rtimes H_1\\ &\simeq (\text{triv}^{H_1}((\text{triv}^{H_2}(-) \otimes *_{H_2}\mathbb{C}) \rtimes H_2) \otimes *_{H_1}\mathbb{C}) \rtimes H_1\\& \xrightarrow{H_2 \to *} \text{triv}^{H_1}((\text{triv}^{H_2}(-) \rtimes H_2) \otimes *_{H_1}\mathbb C) \rtimes H_1\\& \xrightarrow{H_1 \to *} \text{triv}^{G}(- \rtimes G)
\end{align*}

Therefore, the adjoint morphism $c_G\colon L_0M_0(-) \to N_0(-)$ is equivalent to the composite 
\[\begin{tikzcd}
	{L_1L_2M_2M_1(-)} && {L_1N'_2M_1(-)} & {L_1\widehat{\text{Res}}_{H_1}F(-)} && {N_1F(-)} & {N_0(-)}
	\arrow["{L_1c_{H_2}M_1}", from=1-1, to=1-3]
	\arrow["\sim", from=1-3, to=1-4]
	\arrow["{c_{H_1}}", from=1-4, to=1-6]
	\arrow["\sim", from=1-6, to=1-7]
\end{tikzcd}\]
where the identifications above are by Lemma \ref{prod}.

The induction hypothesis implies that each individual map above induces an equivalence on $p$-completed $K$-theory (we also use that $(-)_p$ commutes with colimits, being a left adjoint functor), and therefore, so does $c_G$. This completes the proof.
\end{proof}
\begin{prop}\label{4.7}
    $-\rtimes G\colon KK^G \to KK^{\hat G}$ restricts to a functor $-\rtimes G\colon KK^{G,ft} \to KK^{\hat G,ft}$.
\end{prop}
\begin{proof}
    Let $A$ be an object of $KK^{G}$, and suppose that $K(A)_p$ is of finite type. We have to show that $K(A \rtimes G)_p$ is also of finite type.

    By Theorem \ref{4.5}, we have that $K(A \rtimes G)_p \simeq (K(A)_{hG})_{p}$, and we conclude by item 2 of Proposition \ref{4.1}.
\end{proof}
\begin{kor}\label{hari}
    $- \rtimes G\colon KK^{G,ft} \to KK^{\hat G,ft}$ is an equivalence.
\end{kor}
\begin{proof}
    This is immediate from the statement of Proposition \ref{4.7} applied to both $G$ and $\hat G$, and Corollary \ref{3.1.1}.
\end{proof}
\begin{lem}
    Let $\hyperlink{cast}{\mathbf{B}}$ be the $(G \times \hat G)$-$C^{*}$-algebra defined just before Proposition \ref{3.3}. We have an isomorphism $K(\mathbf{B})_p \simeq \hyperlink{M_E}{\textup{\textbf{M}}}$ in $\textup{Mod}(KU_p[G \times \hat G])$.
\end{lem}
\begin{proof}
    We will show that $K(\mathbf{B}) \simeq \hyperlink{cat}{\textbf{M'}}$, so that they are equivalent after $p$-completion.
    
    Our first step is to establish a unitary equivalence of $C^{*}$-categories $\text{Hilb}{(\mathbb C)}^{\text{fg,proj}} \simeq \text{Hilb}{(K(L^2(G)))}^{\text{fg,proj}}$.
    
     Every linear operator on $L^2(G)$ is compact because $L^2(G)$ is finite-dimensional. Let $S$ be the $K(L^2(G))$-module $L^2(G)$. It is a standard fact that every finitely generated module over $K(L^2(G))$ is of the form $\oplus_{i=1}^{n}S$, for some $n \geq 0$. Thus the underlying category of $\text{Hilb}{(K(L^2(G)))}^{\text{fg,proj}}$ coincides with that of the finitely generated modules over $K(L^2(G))$. One also sees that the underlying category of $\text{Hilb}{(\mathbb C)}^{\text{fg,proj}}$ is that of the finite-dimensional $\mathbb C$-vector spaces. In what follows, we will simply write $\text{Mod}(\mathbb C)$ and $\text{Mod}(K(L^2(G)))$ for $\text{Hilb}{(\mathbb C)}^{\text{fg,proj}}$ and $\text{Hilb}{(K(L^2(G)))}^{\text{fg,proj}}$, respectively.
    
    Let $u\colon\mathbb C \hookrightarrow K(L^2(G))$ be the left upper corner inclusion. This determines a (unital) functor $u_*\colon\text{Mod}{(\mathbb C)} \to \text{Mod}{(K(L^2(G)))}$, given by $M \mapsto M \otimes_{\mathbb C} K(L^2(G))$. We claim that $u_*$ is a unitary equivalence - let $u^{*}\colon \text{Mod}(K(L^2(G))) \to \text{Mod}(\mathbb C)$ be the functor given by $M \mapsto u(\mathbb C)M$; then one checks that $u_* \circ u^{*}$ and $u^*\circ u_*$ are unitarily isomorphic to the respective identity functors.
    
    Using the functor $L_W(u)$, we can transport the $(G \times \hat G)$-action on $\text{Mod}(\mathbb C)$ given by the cocycle $E$ to one on $\text{Mod}(K(L^2(G)))$. We represent this action on $\text{Mod}(K(L^2(G)))$ as a functor $\Psi'\colon B(G \times \hat G) \to N_h(C^{*}\text{Cat})$.
By construction, $K^{\mathcal{C}}(\text{Mod}(K(L^2(G)))$ with this action is equivalent to $M_E$ as a $KU[G \times \hat G]$-module.

     Let $P$ be the $(G \times \hat G)$-action on $K(L^2(G))$ that gives it the structure of the $(G \times \hat G)$-$C^*$-algebra $\textbf{B}$. With notation as in the proof of Lemma \ref{3.2}, for $(g,\chi) \in G \times \hat G$, the $C^{*}$-algebra homomorphism $P_{(g,\chi)}\colon K(L^2(G)) \to K(L^2(G))$ is given by $A \mapsto S_{\chi}^{-1} T_gAT_g^{-1} S_{\chi}$. 
     
     Consider the functor $j\colon C^{*}\text{Alg} \to C^{*}\text{Cat}$ given by $A \mapsto \text{Hilb}(A)^{\text{fg,proj}}$. We obtain an action $j(P)$ on $\text{Mod}(K(L^2(G))$. On further composing with $L_W$, we obtain a functor $\Psi\colon B(G \times \hat G) \to N_h(C^{*}\text{Cat})$.
     
     The functor $\Psi$ is determined by the data of the object $\text{Mod}(K(L^2(G)))$, and for each $(g,\chi)$, the endofunctor $\Psi_{(g,\chi)}\colon \text{Mod}(K(L^2(G))) \to \text{Mod}(K(L^2(G)))$ given by $M \mapsto M_{P_{(g,\chi)}}$, where $M_{P_{(g,\chi)}}$ is the module with the same underlying abelian group as $M$ and $K(L^2(G))$-action twisted by the isomorphism $P_{(g,\chi)}$. Because $K$-theory inverts Morita equivalences, $K^{\mathcal{C}}(\text{Mod}(K(L^2(G)))$ with the resulting $(G \times \hat G)$-action is isomorphic to $K(\textbf{B})$ as a $KU[G \times \hat G]$-module.

    We will now construct an equivalence between the functors $\Psi'$ and $\Psi$. By the observations above, this would complete the proof. 
    
    We start by defining invertible natural transformations $\alpha_{(g,\chi)}$ between the identity functor and $\Psi_{(g,\chi)}$. For $M \in \text{Mod}(K(L^2(G)))$, we define a module isomorphism $\alpha_{(g,\chi)}\colon M \to \Psi_{(g,\chi)}(M)$ by the rule $m \mapsto S_{\chi}^{-1} T_g m$. To make sense of this formula, note that the underlying abelian groups of $\Psi_{(g,\chi)}(M)$ and $M$ are the same. One checks that these maps indeed assemble to an invertible natural transformation $\alpha_{(g,\chi)}$. Stated differently, we have the following invertible $2$-cell in $N_{h}(C^{*}\text{Cat})$:
\[\begin{tikzcd}
	{\text{Mod}(K(L^2(G)))} &&& {\text{Mod}(K(L^2(G)))}
	\arrow[""{name=0, anchor=center, inner sep=0}, "{\Psi_{(g,\chi)}}"', curve={height=30pt}, from=1-1, to=1-4]
	\arrow[""{name=1, anchor=center, inner sep=0}, "{\text{id}}", curve={height=-30pt}, from=1-1, to=1-4]
	\arrow["{\alpha_{(g,\chi)}}", shorten <=8pt, shorten >=8pt, Rightarrow, from=1, to=0]
\end{tikzcd}\]
One readily checks that $\alpha_{(g,\chi)}\circ \alpha_{(g',\chi')} = E(g,\chi,g',\chi')\alpha_{(gg',\chi \chi')} = \chi'(g)\alpha_{(gg',\chi \chi')}$.

We now define a homotopy $\sigma\colon B(G \times \hat G) \times \Delta^{1} \to N_h(C^{*}\text{Cat})$. To do so, we view both the source\footnote{Here, we consider $ B(G \times \hat G)$ as a category with a single object.} and target as simplicial sets; the former is the nerve of a $1$-category, and the latter is the nerve of a $2$-category (as remarked earlier). The source is consequently $2$-coskeletal and the target $3$-coskeletal (see Theorem 5.2 of \cite{2cat}\footnote{The author thanks Lyne Moser for pointing out this reference.}). We will define $\sigma$ as a map of simplicial sets, and start by describing $\sigma$ in low dimensions:

A word on notation - we will denote $n$-simplices, for $n \in \{0,1,2,3\}$, in $\Delta^{1}$ by sequences $x_0...x_{n}$, with $x_i \in \{0,1\}$ and $x_i \leq x_{i+1}$. We will write $\mathcal{C}$ for $\text{Mod}(K(L^2(G)))$. 
\begin{itemize}
    \item \textbf{Dimension 0} Set $\sigma((*,0)) = \mathcal{C}$ and $\sigma((*,1)) = \mathcal{C}$.
    \item \textbf{Dimension 1} Set $\sigma(((g,\chi),00)) = \text{id}_{\mathcal{C}}$, $\sigma(((g,\chi),01)) = \Psi_{(g,\chi)}$ and \\$\sigma(((g,\chi),11)) = \Psi_{(g,\chi)}$.
    \item \textbf{Dimension 2} Set $\sigma(((g,\chi),(g',\chi'),000))$ to be 
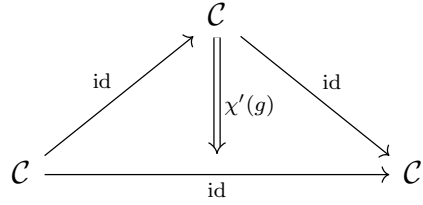
\[\begin{tikzcd}
	&& {\mathcal{C}} \\
	\\
	{\mathcal{C}} &&&& {\mathcal{C}}
	\arrow["{\text{id}}", from=3-1, to=1-3]
	\arrow["{\text{id}}", from=1-3, to=3-5]
	\arrow[""{name=0, anchor=center, inner sep=0}, "{\text{id}}"', from=3-1, to=3-5]
	\arrow["{\chi'(g)}", shorten >=7pt, Rightarrow, from=1-3, to=0]
\end{tikzcd}\]
Set $\sigma(((g,\chi),(g',\chi'),001))$ to be  
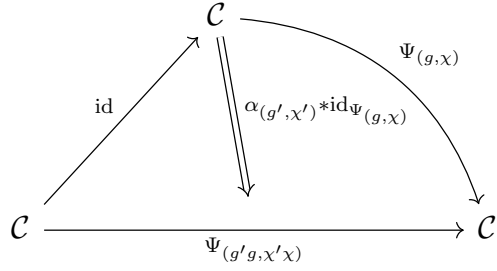
\[\begin{tikzcd}
	&& {\mathcal{C}} \\
	\\
	\\
	{\mathcal{C}} &&&&& {\mathcal{C}}
	\arrow["{\text{id}}", from=4-1, to=1-3]
	\arrow[""{name=0, anchor=center, inner sep=0}, "{\Psi_{(g'g,\chi'\chi)}}"', from=4-1, to=4-6]
	\arrow["{\Psi_{(g,\chi)}}", curve={height=-30pt}, from=1-3, to=4-6]
	\arrow["{\alpha_{(g',\chi')}*\text{id}_{\Psi_{(g,\chi)}}}", shorten >=12pt, Rightarrow, from=1-3, to=0]
\end{tikzcd}\]
Set $\sigma(((g,\chi),(g',\chi'),011))$ to be 
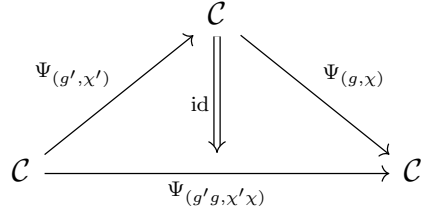
\[\begin{tikzcd}
	&& {\mathcal{C}} \\
	\\
	{\mathcal{C}} &&&& {\mathcal{C}}
	\arrow["{\Psi_{(g',\chi')}}", from=3-1, to=1-3]
	\arrow["{\Psi_{(g,\chi)}}", from=1-3, to=3-5]
	\arrow[""{name=0, anchor=center, inner sep=0}, "{\Psi_{(g'g,\chi'\chi)}}"', from=3-1, to=3-5]
	\arrow["{\text{id}}"', shorten >=7pt, Rightarrow, from=1-3, to=0]
\end{tikzcd}\]
We set  $\sigma(((g,\chi),(g',\chi'),111))$ similarly: 
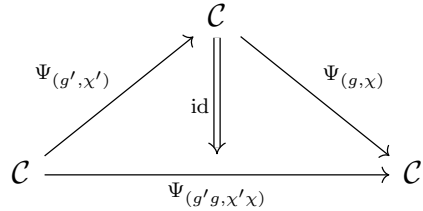
\[\begin{tikzcd}
	&& {\mathcal{C}} \\
	\\
	{\mathcal{C}} &&&& {\mathcal{C}}
	\arrow["{\Psi_{(g',\chi')}}", from=3-1, to=1-3]
	\arrow["{\Psi_{(g,\chi)}}", from=1-3, to=3-5]
	\arrow[""{name=0, anchor=center, inner sep=0}, "{\Psi_{(g'g,\chi'\chi)}}"', from=3-1, to=3-5]
	\arrow["{\text{id}}"', shorten >=7pt, Rightarrow, from=1-3, to=0]
\end{tikzcd}\]
\end{itemize}
One easily checks that these definitions are compatible with the face and degeneracy maps wherever defined, i.e. on simplices of dimension $\leq 2$.

The $3$-coskeletality of $N_h(C^{*}\text{Cat})$ implies that we only have to check the following condition to see that the definition above extends uniquely to the desired homotopy $\sigma$:

\textit{For two composable morphisms $f$ and $f'$ in $B(G \times \hat G) \times \Delta^{1}$, denote by $\sigma_{f,f'}$ the image of the identity $2$-simplex on $f,f'$ and $f'\circ f$. Let $f_3,f_2,f_1$ be a sequence of composable morphisms in $B(G \times \hat G) \times \Delta^{1}$. Then there is an equality of $2$-simplices}$$\sigma_{f_1,f_3 \circ f_2}*(\sigma_{f_2,f_3}\circ \text{id}_{f_1}) = \sigma_{f_2 \circ f_1,f_3}*(\text{id}_{f_3} \circ \sigma_{f_1,f_2})$$
We check this condition case-by-case on the sequence of composable morphisms, according to the underlying sequence after projecting to $\Delta^1$:

\begin{itemize}
    \item The case $0000$ amounts to checking that $E$ is a cocycle.
    \item The case $0001$ reduces to verifying the equality, for $(g_3,\chi_3), (g_2,\chi_2), (g_1,\chi_1) \in G \times \hat G$, $$(\alpha_{(g_2,\chi_2)} \circ \alpha_{(g_1,\chi_1)})*\text{id}_{\Psi_{(g_3,\chi_3)}} = \chi_1(g_2)\alpha_{(g_1g_2,\chi_1\chi_2)}*\text{id}_{\Psi_{(g_3,\chi_3)}}$$This follows from the observation made just before the definition of $\sigma$.
    \item The case $0011$ is just the trivial equality $$\alpha_{(g_1,\chi_1)}*\text{id}_{\Psi_{(g_2g_3,\chi_2\chi_3)}} = \alpha_{(g_1,\chi_1)}*\text{id}_{\Psi_{(g_2g_3,\chi_2\chi_3)}}$$
    \item The cases $0111$ and $1111$ are trivial because all the $2$-simplices involved are identities.
\end{itemize}

Therefore, $\sigma$ defines a homotopy between $\Psi' = \sigma|_{0}$ and $\Psi = \sigma|_{1}$. It is, in fact, a homotopy equivalence, because it factors through the maximal groupoid of $N_h(C^{*}\text{Cat})$.

This shows that the two actions are equivalent. Hence, $K(\textbf{B})$ with its natural $(G \times \hat G)$-action is equivalent to $KU$ with $(G \times \hat G)$-action given by the cocycle $E$, whence it follows that $K(\textbf{B}) \simeq \textbf{M'}$, as desired. 
\end{proof}
We are now ready to state and prove our main comparison result (Theorem \ref{trid}).
\begin{theorem}\label{4.9}
    There is a natural equivalence filling the square of functors 
\[\begin{tikzcd}
	{KK^{G,ft}} &&& {KK^{\hat G,ft}} \\
	\\
	\\
	{\textup{Mod}(KU_p[G])^{ft}} &&& {\textup{Mod}(KU_p[\hat G])^{ft}}
	\arrow["{-\rtimes G}" ,"\sim"', from=1-1, to=1-4]
	\arrow["{r_G}"', from=1-1, to=4-1]
	\arrow["{r_{\hat G}}", from=1-4, to=4-4]
	\arrow["\sim","{(-\otimes_{KU_p[G]}\textup{\textbf{M}})_p}"', from=4-1, to=4-4]
\end{tikzcd}\]
\end{theorem}
\begin{proof}
    The down-right composite of the diagram is the functor $A \mapsto (K(A)_p \otimes_{KU_p[G]} \textbf{M})_p$, with $\hat G$-action induced by that on $\hyperlink{M_E}{\textbf{M}}$. By the previous lemma, we have a natural isomorphism of $KU_p[\hat G]$-modules
    $$(K(A)_p \otimes_{KU_p[G]} \textbf{M})_p \simeq (\text{colim}_{BG}(K(A)_p \otimes_{KU_p} K(\textbf{B})_p))_p$$
    The symmetric monoidality of $(-)_p$ gives a natural isomorphism of $KU_p[\hat G]$-modules
    $$(\text{colim}_{BG}(K(A)_p \otimes_{KU_p} K(\textbf{B})_p))_p \simeq (\text{colim}_{BG} (K(A) \otimes_{KU} K(\textbf{B}))_p)_p$$
    By the K{\"u}nneth theorem ($K(\textbf{B}) \simeq KU$) and the left-adjointness and idempotence of $(-)_p$, we obtain yet another natural isomorphism\footnote{The colimit on the right hand side is to be interpreted in $p$-complete $KU_p$-modules.} of $KU_p[\hat G]$-modules $$ (\text{colim}_{BG} (K(A) \otimes_{KU} K(\textbf{B}))_p)_p \simeq \text{colim}_{BG}(K(A \otimes \textbf{B})_p)$$
    Lastly, Theorem \ref{4.5} with another application of the left-adjointness of $(-)_p$ gives a natural $KU_p[\hat G]$-module isomorphism
    $$\text{colim}_{BG}(K(A \otimes \textbf{B})_p) \simeq K((A \otimes \mathbf{B}) \rtimes G)_p$$
    By Theorem \ref{3.4}, we can replace the upper horizontal arrow by the descended functor $F$ of Proposition \ref{3.3}. We recall that $F$ is given by $A \mapsto (A \otimes \mathbf{B}) \rtimes G$, with $\hat G$-action on the target induced by that on $\hyperlink{cast}{\textbf{B}}$. The composite $r_{\hat G} \circ F$ is also $A \mapsto K((A \otimes \mathbf{B}) \rtimes G)_p$, completing the proof.
\end{proof}
\printbibliography

@misc{lurie,
    title = {Ambidexterity in $K(n)$-Local Stable Homotopy Theory},
    author = {Michael Hopkins and Jacob Lurie},
    year = {2013}
}

@article{bunke2019homotopy,
    author = {Ulrich Bunke} ,
    title = {Homotopy theory with *categories},
    journal = {Theory Appl. Categ. 34 (2019), no. 27, 781-853} ,
    year = {2019}
}

@misc{bunke2023ktheory,
      title={$K$-theory of crossed products via homotopy theory}, 
      author={Ulrich Bunke},
      year={2023},
      eprint={2311.06562},
      archivePrefix={arXiv},
      primaryClass={math.OA}
}

@book{bunke2020homotopy,
    author = {Ulrich Bunke and Alexander Engel},
    title = {Homotopy theory with bornological coarse spaces},
    publisher = {Springer, Lecture Notes in Mathematics, vol.2269},
    year = {2020}
}

@misc{bunke2021additive,
      title={Additive C*-categories and K-theory}, 
      author={Ulrich Bunke and Alexander Engel},
      year={2021},
      eprint={2010.14830},
      archivePrefix={arXiv},
      primaryClass={math.KT}
}

@misc{bunke2023stable,
      title={A stable $\infty$-category for equivariant $KK$-theory}, 
      author={Ulrich Bunke and Alexander Engel and Markus Land},
      year={2023},
      eprint={2102.13372},
      archivePrefix={arXiv},
      primaryClass={math.OA}
}

@misc{bunke2023survey,
      title={A survey on operator $K$-theory via homotopical algebra}, 
      author={Ulrich Bunke and Markus Land and Ulrich Pennig},
      year={2023},
      eprint={2311.17191},
      archivePrefix={arXiv},
      primaryClass={math.OA}
}

@misc{treumann2015representations,
      title={Representations of finite groups on modules over K-theory (with an appendix by Akhil Mathew)}, 
      author={David Treumann},
      year={2015},
      eprint={1503.02477},
      archivePrefix={arXiv},
      primaryClass={math.RT}
}

@article{2cat,
      author = {Ross Street},
      title = {The algebra of oriented simplexes},
      journal = {Journal of Pure and Applied Algebra, vol.49(3)},
      year = {1987}
}

@article{ambrcat,
    author = {Ivo Dell'Ambrogio},
    title = {The unitary symmetric monoidal model category of small $C^{*}$-categories},
    journal = {Homology, Homotopy and Applications, vol.14(2)},
    year = {2012}
}

@misc{bunke2024etheory,
      title={$E$-theory is compactly assembled}, 
      author={Ulrich Bunke and Benjamin Duenzinger},
      year={2024},
      eprint={2402.18228},
      archivePrefix={arXiv},
      primaryClass={math.KT}
}

@book{algwill,
    author = {Dana P. Williams},
    title = {Crossed Products of $C^{*}$-Algebras},
    publisher = {Number 134 in Math. Surveys and Monographs, Amer. Math. Soc., Providence, RI},
    year = {2007}
}

@article{meyerkth,
    author = {Ralf Meyer},
    title = {Categorical aspects of bivariant $K$-theory},
    journal = {K-theory and Noncommutative Geometry (Valladolid 2006), EMS Ser. Congr. Rep., pp. 1-39},
    year = 2007
}

@article{bunke2024kk,
    author = {Ulrich Bunke} ,
    title = {KK- and E-theory via homotopy theory},
    journal = {Orbita Mathematicae, vol.1(2)},
    year = {2024}
}

@article{completion,
    author = {Michael F. Atiyah and Graeme B. Segal} ,
    title = {Equivariant K-theory and completion},
    journal = {Journal of Differential Geometry, 3},
    year = {1969}
}

@article{classically,
    author = {Saad Baaj and Georges Skandalis},
    title = {C*-algèbres de Hopf et théorie de Kasparov équivariante},
    journal = {K-Theory, 2, no.6},
    year = {1989}
}

@misc{kerodon,
  author       = {Jacob Lurie},
  title        = {Kerodon},
  howpublished = {\url{https://kerodon.net}},
  year         = {2024},
}

@book{ha,
    author = {Jacob Lurie},
    title = {Higher Algebra},
    year = {2017},
    url = {https://people.math.harvard.edu/~lurie/papers/HA.pdf}
}
\end{document}